\documentclass[11pt]{amsart}

\usepackage{amssymb, amstext, amscd, amsmath, color}

\usepackage{url}
\usepackage{tikz}
\usepackage{epstopdf}
\usepackage{amsfonts}
\usepackage{amsthm}
\usepackage{amsfonts}
\usepackage{array}
\usepackage{epsfig}
\usepackage{eucal}
\usepackage{latexsym}
\usepackage{mathrsfs}
\usepackage{textcomp}
\usepackage{verbatim}
\usepackage{setspace}
\usepackage{enumerate}
\usepackage[colorlinks=true,linkcolor=blue,urlcolor=blue,citecolor=red]{hyperref}

\newtheorem{thm}{Theorem}[section]

\newtheorem{cor}[thm]{Corollary}

\newtheorem{example}[thm]{Example}

\newtheorem{lem}[thm]{Lemma}

\newtheorem{rem}[thm]{Remark}

\numberwithin{equation}{section}

\newcommand{\tr}{\mathrm{tr}}

  \newcommand{\R}{{\mathbb{R}}}

  \newcommand{\U}{{\mathcal{U}}}

  \newcommand{\Z}{{\mathbb{Z}}}
\newcommand{\Aut}{\operatorname{Aut}}

\newcommand{\rank}{\operatorname{rank}}

\newcommand{\Bk}{\Bar{k}}

\usepackage{mathtools}

\def\quotient#1#2{
    \raise1ex\hbox{$#1$}\big/\lower1ex\hbox{$#2$}}
\def\Bigquotient#1#2{
    \raise1ex\hbox{$#1$}\Big/\lower1ex\hbox{$#2$}}
\usepackage{upgreek}
\tikzset{math3d/.style= {x={(1cm,0cm)}, y={(0.353cm,0.353cm)}, z={(0cm,1cm)}}}
\usepackage{tikz-3dplot}

\usepackage{blkarray}

\tikzstyle{vertex}=[circle, draw, fill=black, inner sep=0pt, minimum size=3pt]
\tikzstyle{fadedvertex}=[vertex, fill=black!30]
\tikzstyle{edge}=[line width=1pt]
\tikzstyle{fadededge}=[edge,color=black!30]

\begin{document}

\title[Rigidity of symmetric linearly constrained frameworks in the plane]{Rigidity of symmetric linearly constrained frameworks in the plane}
\author[Anthony Nixon]{Anthony Nixon}
\address{Dept.\ Math.\ Stats.\\ Lancaster University\\
Lancaster LA1 4YF \\U.K.}
\email{a.nixon@lancaster.ac.uk}
\author[Bernd Schulze]{Bernd Schulze}
\address{Dept.\ Math.\ Stats.\\ Lancaster University\\
Lancaster LA1 4YF \\U.K.}
\email{b.schulze@lancaster.ac.uk}
\author[Joseph Wall]{Joseph Wall}
\address{Dept.\ Math.\ Stats.\\ Lancaster University\\
Lancaster LA1 4YF \\U.K.}
\email{j.wall@lancaster.ac.uk}
\thanks{2010 {\it  Mathematics Subject Classification.}
52C25; 05C70; 20C35.\\
Key words and phrases: rigidity, linearly constrained framework, incidental symmetry, symmetric framework, recursive construction}

\begin{abstract}
A bar-joint framework $(G,p)$ is the combination of a finite simple graph $G=(V,E)$ and a placement $p:V\rightarrow \mathbb{R}^d$. The framework is rigid if the only edge-length preserving continuous motions of the vertices arise from isometries of the space. Motivated by applications where boundary conditions play a significant role, one may generalise and consider linearly constrained frameworks where some vertices are constrained to move on fixed affine subspaces. Streinu and Theran  characterised exactly which linearly constrained frameworks are generically rigid in 2-dimensional space. In this article we extend their characterisation to symmetric frameworks.
In particular necessary combinatorial conditions are given for a symmetric linearly constrained framework in the plane to be isostatic (i.e. minimally infinitesimally rigid) under any finite point group symmetry. 
In the case of rotation symmetry groups whose order is either 2 or odd, these conditions are then shown to be sufficient under suitable genericity assumptions, giving precise combinatorial descriptions of symmetric isostatic graphs in these contexts.
\end{abstract}

\date{}
\maketitle

\section{Introduction}

A (bar-joint) \emph{framework} $(G,p)$ is the combination of a finite simple graph 
$G=(V,E)$ and a map $p:V\rightarrow \mathbb{R}^d$ which assigns positions to the vertices, and hence lengths to the edges. With stiff bars for the edges and full rotational freedom for the joints representing the vertices, the topic of rigidity theory concerns whether the framework may be deformed without changing the graph structure or the bar lengths. While `trivial' motions are always possible due to actions of the Euclidean isometry group, the framework is \emph{flexible} if  a non-trivial motion is possible and \emph{rigid} if no non-trivial motion exists. Applications of rigidity abound, see \cite{cg} inter alia.

The problem of determining whether a given framework is rigid is computationally difficult \cite{Abb}. However, every graph has a typical behaviour in the sense that either all `generic' (i.e. almost all) frameworks with the same underlying graph are rigid or all are flexible. So, generic rigidity depends only on the graph and is often studied using a linearisation known as infinitesimal rigidity, which is equivalent to rigidity for generic frameworks  
\cite{AR}.
On the real line it is a simple folklore result that rigidity coincides with graph connectivity. In the plane a celebrated theorem due to Polaczek-Geiringer \cite{PG}, often referred to as Laman's theorem due to a rediscovery in the 1970s \cite{laman}, characterises the generically rigid graphs precisely in terms of graph sparsity counts, and these combinatorial conditions can be checked in polynomial time. Combinatorial characterisations of generically rigid graphs in dimension $3$ or higher have not yet been found.

A linearly constrained framework is a bar-joint framework in which certain vertices are constrained to lie in given affine subspaces, in addition to the usual distance constraints between pairs of vertices. Linearly constrained frameworks arise naturally in practical applications where objects may  be constrained to move, for example, on the ground, along a wall, or in a groove. In particular, slider joints are very common in mechanical engineering (see, e.g., \cite{EJNSTW}) and have been applied to modelling boundary conditions in biophysics \cite{Tetal}.  

Streinu and Theran \cite{STh} proved a characterisation of generic rigidity for linearly constrained frameworks in $\R^2$. The articles \cite{CGJN,JNT} provide an analogous  characterisation for generic rigidity of
linearly constrained frameworks in $\R^d$ as long as the dimensions of the  affine subspaces at each vertex are sufficiently small (compared to $d$), and \cite{GJN} characterises the stronger notion of global rigidity in this context.

Separately, the genericity hypothesis, while natural from an algebraic geometry viewpoint, does not apply in many practical applications of rigidity theory. In particular,  structures in mechanical and structural engineering, computer-aided design, biophysics, and materials science often exhibit non-trivial symmetry. This has motivated multiple groups of researchers to study the rigidity of symmetric structures over the last two decades, which has led to an explosion of results in this area. We direct the reader to \cite{cg,SW2} for details. Importantly, there are two quite different notions of symmetric rigidity that one may consider. Firstly, \emph{forced symmetric rigidity} concerns frameworks that are symmetric and only motions that preserve the symmetry are allowed (that is, a framework may be flexible but if all the non-trivial motions destroy the symmetry then it is still  `forced symmetric rigid'). Secondly, \emph{incidental symmetric rigidity} again concerns symmetric frameworks, but the question of whether they are rigid is the same as in the non-symmetric case. 

It is incidental symmetry that we focus on in this article. More specifically, we are interested in describing, combinatorially, when a generic symmetric linearly constrained framework in the plane is \emph{isostatic}, i.e. minimally infinitesimally rigid in the sense that it is infinitesimally rigid but ceases to be so after deleting any edge. The corresponding question for standard bar-joint frameworks in the plane has been studied in  \cite{schulze,BS4}. In these papers, 
Laman-type results have been established for the groups generated by a reflection, the half-turn and a three-fold rotation, but these problems remain open for the other groups that allow isostatic frameworks.

In this article we give the first rigidity-theoretic analysis of symmetric linearly constrained frameworks. 
Our main results are  representation-theoretic necessary conditions for isostaticity for all relevant symmetry groups (Section~\ref{sec:necrep}), as well as complete combinatorial characterisations of symmetry-generic isostatic frameworks for the groups generated by a rotation of order either 2 or odd. These results, Theorems \ref{thm:mainc2} and \ref{thm:maincodd}, are proved in Sections \ref{sec:C2} and \ref{sec:Cn}.
The proofs of these combinatorial characterisations rely on symmetry-adapted Henneberg-type recursive construction moves described in Section~\ref{sec:ops}. 

In \cite{NSW} we analysed symmetric frameworks all of whose points are constrained to lie on a single surface. This situation can be modelled by linearly constrained frameworks where the linear constraints are chosen to represent the normals to the given surface. From this viewpoint one may wonder about the case when the surface is sufficiently generic. The results of this paper point in this direction but focus on the 2-dimensional case.
Moreover the present setting is more general in the sense that the set of the framework points that have associated linear constraints is not pre-determined.

\section{Rigidity theory}\label{sec:basics}

Throughout this paper we will use the term {\em graph} to describe a graph which may contain
multiple edges and loops and denote such a graph by $G=(V,E,L)$ where $V,E,L$ are the sets of vertices, edges and loops, respectively. We will use the terms {\em simple graph} to describe a graph which contains neither multiple edges nor loops, and {\em looped simple graph} to describe a graph which contains no multiple edges but may contain (multiple) loops.

\subsection{Linearly constrained frameworks}

Following \cite{CGJN}, we define a {\em linearly constrained framework in $\R^d$}  to be a triple $(G,
p, q)$ where $G=(V,E,L)$ is a looped simple graph,
$p:V\to \R^d$ is injective and
$q:L\to \R^d$. For $v_i\in V$ and $\ell_j\in L$ we put $p(v_i)=p_i$ and
$q(\ell_j)=q_j$.

An {\em infinitesimal motion} of $(G, p, q)$ is a map $\dot
p:V\to \R^d$ satisfying the system of linear equations:
\begin{eqnarray}
\label{eqn1} (p_i-p_j)\cdot (\dot p_i-\dot p_j)&=&0 \quad \mbox{ for all $v_iv_j \in E$}\\
\label{eqn2} q_j\cdot \dot p_i&=&0 \quad \mbox{ for all incident pairs $v_i\in V$ and
$\ell_j \in L$.}
\end{eqnarray}
The second constraint implies that  the infinitesimal velocity of each
$v_i\in V$ is constrained to lie on the hyperplane through $p_i$ with normal vector $q_j$
for each loop $\ell_j$ incident to $v_i$.

The {\em rigidity matrix $R (G, p, q)$} of the linearly constrained framework $(G, p, q)$ is  the
matrix of coefficients of this system of equations for the unknowns
$\dot p$. Thus $R (G, p, q)$ is a $(|E|+|L|)\times d|V|$ matrix, in
which: the row indexed  by an edge $v_iv_j\in E$ has $p_i-p_j$ and
$p_j-p_i$ in the $d$ columns indexed by $v_i$ and $v_j$,
respectively and zeros elsewhere; the row indexed  by a loop
$\ell_j=v_iv_i\in L$ has $q_j$  in the $d$ columns indexed by $v_i$ and
zeros elsewhere. The $|E|\times d|V|$ sub-matrix consisting of the rows indexed by $E$ is the {\em bar-joint rigidity matrix} $R(G-L,p)$ of the bar-joint framework $(G-L,p)$.

The framework $(G, p, q)$ is {\em infinitesimally rigid} if its only
infinitesimal motion is $\dot p =0$, or equivalently if $\rank R(G,
p, q) = d|V|$.   A framework $(G,p,q)$  is called \emph{isostatic} if it is infinitesimally rigid and \emph{independent} in the sense that the rigidity matrix of $(G,p,q)$ has no non-trivial row dependence.
Equivalently, $(G,p,q)$ is isostatic if it is infinitesimally rigid and deleting any single edge results in a framework that is not infinitesimally rigid.

\begin{example}
Suppose that $G$ consists of a single vertex with one loop and $(G,p,q)$ is a framework in $\mathbb{R}^2$. Then $G$ is not infinitesimally rigid since the translation along the line corresponding to the loop is an infinitesimal motion (and any infinitesimal motion is considered non-trivial in our context). 

Similarly a complete graph, realised generically, is infinitesimally rigid in any dimension as a bar-joint framework. However as a linearly constrained framework it is not infinitesimally rigid since the translations and rotations are infinitesimal motions.
\end{example} 

A linearly constrained framework $(G,p,q)$ in $\R^d$ is {\em generic} if $\textrm{rank } R(G,p, q) \geq \textrm{rank } R(G,p', q') $ for all frameworks $(G,p',q')$ in $\R^d$.

We say that the looped simple graph $G$ is  {\em rigid} in $\R^d$
if $\rank R(G, p, q) = d|V|$ for some realisation $(G,p, q)$ in
$\R^d$, or equivalently if $\rank R(G, p, q) = d|V|$ for all {\em
generic} realisations $(G,p,q)$.  Similarly,  we define  $G$ to be \emph{isostatic (independent)} if there exists
a framework $(G,p,q)$  that is  isostatic (independent). 

Streinu and Theran gave the following  characterisation of looped simple graphs which are rigid in  $\mathbb{R}^2$. We will say that $G=(V,E,L)$ is: \emph{sparse} if $|E'|+|L'|\leq 2|V'|$ for all subgraphs $(V',E')$ of $G$ and $|E'|\leq 2|V'|-3$ for all simple subgraphs with $|E'|>0$; and \emph{tight} if it is sparse and $|E|+|L|=2|V|$.

\begin{thm}\label{thm:lm}\cite{STh}
A generic linearly constrained framework $(G,p,q)$ in $\mathbb{R}^2$ is isostatic if and only if $G$ is tight.
\end{thm}

While the theorem gives a complete answer in the generic case, the present article will extend this to apply under the presence of non-trivial symmetries.

It would be an interesting future project to extend our analysis to higher dimensions. While for bar-joint frameworks little is known when $d\geq 3$, in the linearly constrained case characterisations are known when suitable assumptions are made on the affine subspaces defined by the linear constraints \cite{CGJN,JNT}.

\subsection{Symmetric linearly constrained frameworks}

Let $G= (V,E,L)$ be a looped simple graph and $\Gamma$ be a finite group.
Then the pair $(G,\phi)$ is called $\Gamma$-symmetric if $\phi: \Gamma \to \Aut (G)$ is a homomorphism, where $\Aut (G)$ denotes the automorphism group of $G$.
Note that an automorphism $\phi(\gamma)$ of $G$ consists of a permutation of the vertices, $\phi_1(\gamma)$, and a permutation of the loops, $\phi_2(\gamma)$, so that $v_iv_j\in E$ if and only if $\phi_1(v_i)\phi_1(v_j)\in E$, and  $v_i$ is incident to the loop $l_j$ if and only if $\phi_1(v_i)$ is incident to $\phi_2(l_j)$. The permutation $\phi_1(\gamma)$ clearly induces a permutation of the edges in $E$. Moreover, $\phi_1(\gamma)$ must map a vertex with $n$ loops to another vertex with $n$ loops, and a loop can only be fixed by $\phi_2(\gamma)$ (i.e., $\phi_2(\gamma)(l_j)=l_j$) if it is incident to a vertex $v_i$ that is fixed by $\phi_1(\gamma)$ (i.e., $\phi_1(\gamma)(v_i)=v_i$).  In our context, a loop always represents a linear constraint in the plane. Thus, we will assume that a loop at $v_i$ can only be fixed by $\phi_2(\gamma)$ if $\gamma$ is the identity or an element of order $2$ (since a line in the plane cannot be unshifted  by an isometry of order greater than $2$).

If $\Gamma$ is clear from the context, then  a $\Gamma$-symmetric graph will often simply be called \emph{symmetric}.
Similarly, if $\phi$ is clear from context, we may simply refer to the $\Gamma$-symmetric graph $(G,\phi)$ by $G$.

Let $(G,\phi)$ be a $\Gamma$-symmetric looped simple graph.
Then, for a homomorphism $\tau: \Gamma \to O(\R^{2})$, we say that a linearly constrained framework $(G,p,q)$ is \emph{$\Gamma$-symmetric (with respect to $\phi$ and $\tau$)}, or simply \emph{$\tau(\Gamma)$-symmetric}, if
\begin{itemize}
    \item  $\tau(\gamma)p_{i} = p_{\phi(\gamma)i}$ for all $v_i\in V$ and all $\gamma \in \Gamma$;
    \item $\tau(\gamma)q_{j} = q_{\phi(\gamma)j}$ for all $l_j \in L$ and all $\gamma \in \Gamma$ whose order is not $2$;
    \item $\tau(\gamma)q_{j} = -q_{\phi(\gamma)j}$ if $\tau(\gamma)$ is the half-turn and the loop $l_j$ is fixed by $\gamma\in \Gamma$;
      \item $\tau(\gamma)q_{j} = \pm q_{\phi(\gamma)j}$ if $\tau(\gamma)$ is a reflection and the loop $l_j$ is fixed by $\gamma\in \Gamma$.
\end{itemize}

We will refer to $\tau(\Gamma)$ as a \emph{symmetry group} and to  elements of $\tau(\Gamma)$ as \emph{symmetry operations} or simply \emph{symmetries} of $(G,p,q)$.  

A $\Gamma$-symmetric linearly constrained framework $(G,p,q)$ is \emph{$\Gamma$-generic (with respect to $\tau$ and $\phi$)} if $\textrm{rank }R(G,p,q)\geq \textrm{rank }R(G,p',q')$ for all linearly constrained frameworks $(G,p',q')$ that are $\Gamma$-symmetric with respect to $\tau$ and $\phi$.
The set of all $\Gamma$-generic realisations of $G$ (with respect to $\tau$ and $\phi$) is an open dense subset of the set of all $\Gamma$-symmetric realisations of $G$ (with respect to $\tau$ and $\phi$).
Thus, we may say that a graph $G$ is \emph{$\tau(\Gamma)$-isostatic (independent, infinitesimally rigid, rigid)} if there exists a $\Gamma$-symmetric framework $(G,p,q)$ (with respect to $\tau$ and $\phi$) which is isostatic (independent, infinitesimally rigid, rigid).
Later we will often remove $\phi$ from this notation and simply refer to a  $\tau(\Gamma)$-isostatic (independent, infinitesimally rigid, rigid) graph (where $\phi$ is clear from the context).

Throughout this paper, we will use a version of the Schoenflies notation for symmetry operations and groups.
The relevant symmetry operations in the plane are the identity, denoted by  $\textrm{id}$, rotations by $\frac{2\pi}{n}$, $n\in \mathbb{N}$, about the origin, denoted by $c_n$ and reflections in lines through the origin, denoted by $\sigma$.
The relevant  symmetry groups for this paper are the group $C_s$ generated by a reflection $\sigma$, the cyclic groups $C_n$ generated by a rotation $c_n$ (where $C_1$ is just the trivial group), and the dihedral group $C_{nv}$, which is generated by $c_n$, $n\geq 2$, and $\sigma$. 

\section{Necessary Conditions for Isostatic Linearly-Constrained Frameworks}\label{sec:necrep}

In this section, we will establish necessary conditions for a symmetric linearly constrained framework in the plane to be isostatic. To this end we first show that the rigidity matrix  of a symmetric linearly constrained framework can be transformed into a block-decomposed form by using suitable symmetry-adapted bases. The necessary conditions are then obtained by comparing the number of rows and columns of each submatrix block. Using basic character theory, these conditions can be stated simply in terms of the number of structural components that remain unshifted  under the various symmetries of the framework.

\subsection{Block-diagonalisation of the rigidity matrix}

We need the following basic definitions.

If $A$ is a $m \times n$ matrix and $B$ is a $p \times q$ matrix, the Kronecker product $A\otimes B$ is the $pm \times qn$ block matrix:
\[A\otimes B = \left[
\begin{tabular}{ c c c }
        $b_{11}A$ & $\dots$ & $b_{1q}A$\\
        $\vdots$ & $\ddots$ & $\vdots$ \\
        $b_{p1}A$ & $\dots$ & $b_{pq}A$ \\
    \end{tabular}
\right].
\]
Let $\Gamma$ be a finite group and let $\tau(\gamma)$ denote the $2 \times 2$ matrix which represents $\gamma$ with respect to the canonical basis of $\R^{2}$. For a given $\Gamma$-symmetric framework $(G,p,q)$ (with respect to $\phi$ and $\tau$) and any $\gamma\in\Gamma$, we let
 $P_V(\gamma)$, $P_E(\gamma)$ and $P_L(\gamma)$ be the permutation matrix of $V$, $E$ and $L$ respectively, induced by  the automorphism $\phi(\gamma)$.
 
Let $P_{L}^*: \Gamma \to O(\R^{|L|})$ be the representation obtained from $P_L$ by replacing each $1$ in the $j$-th diagonal entry of $P_L(\gamma)$ by $-1$ if $\tau(\gamma)q_j = - q_j$ (i.e. if the loop $l_j$ is fixed by $\phi_2(\gamma)$ and either $\tau(\gamma)$ is the half-turn, or $\tau(\gamma)$ is a reflection and the normal $q_j$ of the line corresponding to $l_j$ is perpendicular to the line of reflection).

We may then form the two key representations   $\tau \otimes P_V : \Gamma \to O(\R^{(2|V|)\times (2|V|)})$ and $P_{E,L} := P_E \oplus P_L^* : \Gamma \to O(\R^{(|E|+|L|) \times (|E|+|L|)})$ that are needed to block-decompose the rigidity matrix of $(G,p,q)$.

\begin{lem}\label{lem: Rextu=intz}
Let $G$ be a graph, $\tau(\Gamma)$ be a symmetry group, and $\phi: \Gamma \to \Aut(G)$ be a homomorphism.\\
If $R(G,p,q)u = z$, then for all $\gamma \in \Gamma$, we have $$R(G,p,q)(\tau\otimes P_{V})(\gamma)u = P_{E,L}(\gamma)z.$$
\end{lem}

\begin{proof}
Suppose $R(G,p,q)u = z$.
Fix $\gamma \in \Gamma$ and let $\tau(\gamma)$ be the orthogonal matrix representing $\gamma$ with respect to the canonical basis of $\R^{2}$.
We enumerate the rows of $R(G,p,q)$ by the set $\{a_{1}, \dots, a_{|E|}, b_{1}, \dots, b_{|L|}\}$.
By \cite{BS2}, we know that $(R(G,p)(\tau\otimes P_V)(\gamma)u)_{a_{i}} = (P_{E,L}(\gamma)z)_{a_{i}}$, for all $i \in [|E|]$. We are left to show the result holds for the rows of $R(G,p,q)$ which represent the normal vectors of the vertices with loops.

We label the vertices of $V$ by $v_1,\dots, v_{|V|}$ and the loops of $L$ by $l_1,\dots, l_{|L|}$.
Write $u \in \R^{2|V|}$ as $u = (u_{1}, \dots, u_{|V|})$, where $u_{i} \in \mathbb{R}^2$ for all $v_i \in V$, and let the loop $l_j$ be incident to $v_i$. Furthermore let $\Phi(\gamma)(l_j) = l_k$ and $\Phi(\gamma)(v_i) = v_m$.

\emph{Case (1):} Suppose that $j\neq k$ or that $j=k$ and $\tau(\gamma)q_j=q_j$.
We first see that $(P_{E,L}(\gamma)z)_{b_{k}} = z_{b_{j}}$
by the definition of $P_{L}^*(\gamma)$.
From $R(G,p,q)u = z$, we also get that $z_{b_{j}} = q_{j} \cdot u_{i}$.
Therefore,
\begin{align*}
    (R(G,p,q)(\tau\otimes P_{V})(\gamma)u)_{b_{k}}
    &= (q_{k})_{1} \cdot (\tau(\gamma)u_{i})_{1} + (q_{k})_{2} \cdot (\tau(\gamma)u_{i})_{2}\\
    &= q_{k} \cdot (\tau(\gamma)u_{i})\\
    &= (\tau(\gamma)q_{j}) \cdot (\tau(\gamma)u_{i}).
\end{align*} 

\emph{Case (2):} Suppose that $j=k$ and $\tau(\gamma)q_j=-q_j$.
We now see that $(P_{E,L}(\gamma)z)_{b_{j}} = -z_{b_{j}}$
by the definition of $P_{L}^*(\gamma)$.
From $R(G,p,q)u = z$, we also get that $z_{b_{j}} = q_{j} \cdot u_{i}$. So
\begin{align*}
    (R(G,p,q)(\tau\otimes P_{V})(\gamma)u)_{b_{j}}
    &= (q_{j})_{1} \cdot (\tau(\gamma)u_{i})_{1} + (q_{j})_{2} \cdot (\tau(\gamma)u_{i})_{2}\\
    &= q_{j} \cdot (\tau(\gamma)u_{i})\\
    &= (-\tau(\gamma)q_{j}) \cdot (\tau(\gamma)u_{i}).
\end{align*}
Finally, 
since the canonical inner product on $\R^{2}$ is invariant under the orthogonal transformation $\tau(\gamma) \in O(\mathbb{R}^2)$, in case (1) we obtain 
$$(\tau(\gamma)q_{j}) \cdot (\tau(\gamma)u_{i}) = q_{j} \cdot u_{i} = z_{b_{j}},$$
as desired. Similarly, in case (2) we obtain:
$$(-\tau(\gamma)q_{j}) \cdot (\tau(\gamma)u_{i}) = -q_{j} \cdot u_{i} = -z_{b_{j}}$$
finishing the proof.
\end{proof}

The following is an immediate corollary of Schur's lemma (see e.g. \cite{Serre}) and the lemma above.

\begin{cor}
Let $(G,p,q)$ be a $\tau(\Gamma)$-symmetric framework and let $I_{1}, \dots, I_{r}$  be the  pairwise non-equivalent irreducible linear representations of $\tau(\Gamma)$.
Then there exist matrices $A,B$ such that the matrices $B^{-1}R(G,p,q)A$ and $A^{-1}R(G,p,q)^{T}B$ are block-diagonalised and of the form
\[
\left(\begin{tabular}{ c c c c c }
        $R_{1}$ & & & & $\mathbf{0}$ \\
         & $R_{2}$ & & & \\
         & & $\ddots$ & & \\
         & & & & \\
        $\mathbf{0}$ & & & & $R_{r}$ \\
    \end{tabular} \right)
\]
where the submatrix $R_{i}$ corresponds to the irreducible representation $I_{i}$.
\end{cor}

This block decomposition corresponds to the decompositions $\R^{2|V|} = X_{1} \oplus \cdots \oplus X_{r}$ and $\R^{|E|+|L|} = Y_{1} \oplus \cdots \oplus Y_{r}$.
The space $X_{i}$ is the $(\tau \otimes P_{V})$-invariant subspace of $\mathbb{R}^{2|V|}$ corresponding to $I_i$, and the space $Y_{i}$ is the $P_{E,L}$-invariant subspace of $\mathbb{R}^{|E|+|L|}$ corresponding to $I_i$.
Thus, the submatrix $R_{i}$ has size $(\dim(Y_{i})) \times (\dim(X_{i}))$.

Using the block-decomposition of the rigidity matrix, we may follow the basic approach described in \cite{FGsymmax,BS2} to derive added necessary conditions for a symmetric linearly constrained framework to be isostatic. We first need the following result.

If $A = (a_{ij})$ is a square matrix then the trace of $A$ is given by $\tr(A) = \sum_{i}a_{ii}$.
For a linear representation $\rho$ of a group $\Gamma$ and a fixed ordering $\gamma_1,\ldots, \gamma_{|\Gamma|}$ of the elements of $\Gamma$, the character of $\rho$ is the $|\Gamma|$-dimensional vector $\chi(\rho)$ whose $i$th entry is $\tr(\rho(\gamma_i))$. 

\begin{thm}\label{thm: character count}
Let $(G,p,q)$ be a $\tau(\Gamma)$-symmetric framework.
If $(G,p,q)$ is isostatic, then
$$\chi(P_{E,L}) = \chi( \tau \otimes P_{V}).$$
\end{thm}

\begin{proof} 
Since $(G,p,q)$ is isostatic, the rigidity matrix of $(G,p,q)$ is a non-singular square matrix. 
Thus, by Lemma~\ref{lem: Rextu=intz}, we have $$R(G,p,q)(\tau\otimes P_V)(\gamma)(R(G,p,q))^{-1}= P_{E,L}(\gamma) \quad \textrm{ for all } \gamma\in \Gamma.$$ It follows that  $\tau \otimes P_{V}$ and $P_{E,L}$ are isomorphic representations of $\Gamma$. Hence, 
$$\chi(P_{E,L}) = \chi(\tau \otimes P_{V}).$$
\end{proof}

\subsection{Character table}

We now calculate the characters of the representations appearing in the statement of Theorem~\ref{thm: character count}.
The possible symmetry operations in the plane are
rotations around the origin, denoted by $c_n$, $n\in \mathbb{N}$, and reflections, denoted by  $\sigma$. 

Recall that for a  looped simple graph $G=(V,E,L)$ that is  $\Gamma$-symmetric (with respect to $\phi=(\phi_1,\phi_2)\in \textrm{Aut}(G)$), we say that a vertex $v_i\in V$ is \emph{fixed} by $\gamma\in\Gamma$ if $\phi_1(\gamma)(v_i)=v_i$. Similarly, an edge $v_iv_j\in E$ is \emph{fixed} by $\gamma\in\Gamma$ if both $v_i$ and $v_j$ are fixed by $\gamma$ or if  $\phi_1(\gamma)(v_i)=v_j$ and $\phi_1(\gamma)(v_j)=v_i$.
Similarly, a loop $l_j \in L$ is \emph{fixed} by $\gamma \in \Gamma$ if $\phi_2(l_j)=l_j$.
For groups of order $2$, we will often just say that  
an edge or a loop is \emph{fixed} if it is fixed by the non-trivial group element.
Note  that loops that are fixed by $\gamma$ correspond to lines that are unshifted by $\tau(\gamma)$ and linearly constrain points of $(G,p,q)$ that are also unshifted by $\tau(\gamma)$.
By the definition of $P_{E,L}$, in a matrix $P_{E,L}(\gamma)$ there are two possibilities for the entry of a fixed loop, namely $\pm 1$, depending on whether the normal of the linear constraint is preserved or inverted by $\tau(\gamma)$.

An edge or a loop of $G$ cannot be fixed by an element of $\Gamma$ that corresponds to a rotation $c_n$, $n\geq 3$, so we have a separate column for $c_2$ below.
The number of vertices that are fixed by the element in $\Gamma$ corresponding to the half-turn $c_{2}$, a general $n$-fold rotation $c_n$, or the reflection $\sigma$ are denoted by $v_{2}$, $v_n$ and $v_{\sigma}$, respectively.
The number of edges that are fixed by the element in $\Gamma$ corresponding to the half-turn $c_2$ and a  reflection $\sigma$ are denoted by $e_2$ and $e_\sigma$ respectively.
Finally, the number of loops that are fixed by $c_2$ 
and $\sigma$ is $l_2$
and $l_{\sigma,+},l_{\sigma,-}$ respectively, where $l_{\sigma,+}$ counts linear constraints perpendicular to the mirror whose normals are preserved by the reflection, and $l_{\sigma,-}$ counts those linear constraints parallel to the mirror with normals inverted by the reflection.

This gives all the information necessary to complete Table \ref{tab: CT2dlin}.

\begin{table}[ht]
    \centering
    \begin{tabular}{ |c|c|c|c|c|} 
\hline
 & $\textrm{id}$ & $c_{n\geq3}$ & $c_{2}$ & $\sigma$ \\
\hline
$\chi(P_{E,L})$ & $|E|+|L|$ & 0 & $e_{2}-l_2$ & $e_{\sigma}+l_{\sigma,+}-l_{\sigma,-}$\\
\hline
$\chi(\tau \otimes P_{V})$ & $2|V|$ & $2v_n \cos(\frac{2\pi}{n})$ & $-2v_{2}$ & 0\\
\hline
\end{tabular}
    \caption{Character table for symmetry operations of the plane.}
    \label{tab: CT2dlin}
\end{table}

In the following proofs we shall use Theorem \ref{thm: character count} to draw conclusions from Table~\ref{tab: CT2dlin}.
The first colunn (for the identity element) simply recovers the result from Theorem \ref{thm:lm} that $|E| + |L| = 2|V|$ for an isostatic linearly constrained framework in the plane. The other columns provide further conditions for isostaticity in the presence of symmetry.
It is easy to see that if the counts in Corollary~\ref{lem: element count 2-dim} are satisfied for a $\tau(\Gamma)$-symmetric framework, then the corresponding counts are also satisfied for any $\tau(\Gamma')$-symmetric subframework with $\Gamma'\subseteq \Gamma$. 

\begin{cor}\label{lem: element count 2-dim}
If $(G,p,q)$ is a $\tau(\Gamma)$-symmetric isostatic framework, then the following hold:
\begin{itemize}
    \item if $c_{n}\in \tau(\Gamma)$, where $n$ is odd or $n\geq 6$ is even,  then $v_n = e_n = l_n =0$;
    \item if $c_2 \in \tau(\Gamma)$ then $v_2 = e_2 = l_2 = 0$ or $v_2 = 1, l_2 = 2$;
    \item if $c_4 \in \tau(\Gamma)$ then $v_4 = 0,1$ and $e_4 = l_4 = 0$;
    \item if $\sigma\in \tau(\Gamma)$ then $e_\sigma + l_{\sigma,+} = l_{\sigma,-}$.
\end{itemize}
\end{cor}

\begin{proof}
Recall from Theorem~\ref{thm: character count} that if 
$(G,p,q)$ is a $\tau(\Gamma)$-symmetric isostatic framework, then
$\chi(P_{E,L}) = \chi( \tau \otimes P_{V})$.
We now consider each of the columns in Table~\ref{tab: CT2dlin}.
From the second column, we obtain $v_n \cos(\frac{2\pi}{n}) =0$, hence either $v_n =0$ or $\cos(\frac{2\pi}{n}) = 0$.
The latter is only possible for a positive integer $n$ when $n=4$.
Hence, for any symmetry group containing $c_4$, we have $v_4=0$ or $v_4=1$, since $p$ is injective.
For all other $n\geq 3$, there are no fixed vertices, edges or loops.

In the second column, the equation $\chi(P_{E,L}) = \chi( \tau \otimes P_{V})$ can only hold for $c_{2}\in \tau(\Gamma)$ if $2v_2 = l_2 - e_2$.
Again recalling that for rotations we may have at most one fixed vertex, this implies that either $v_2=0$ then necessarily we have $l_2=0$ and hence $e_2=0$, or  $v_2=1$ then $e_2=l_2-2$.
Considering the condition given by $\textrm{id}$ at the single fixed vertex, $l_2 \leq 2$. This gives when $v_2 = 1$, $l_2 = 2$ and $e_2 = 0$.
Any group containing $c_4$ necessarily contains $c_2$, and the two conditions above are not mutually exclusive.
Instead, a fixed point in a $C_4$-symmetric framework is constrained by two lines which must be perpendicular.

For the reflection, from the table, we immediately have $e_\sigma + l_{\sigma,+} = l_{\sigma,-}$.
\end{proof}

Note that from these symmetry operations \emph{any} symmetry group in the plane is possible. This contrasts with the situation for bar-joint frameworks in the plane where isostatic symmetric frameworks are only possible for a small number of symmetry groups, see \cite{cfgsw,schulze,BS4} for details.

We are now able to use Corollary \ref{lem: element count 2-dim}  to summarize the conclusions about $\tau(\Gamma)$-symmetric isostatic linearly constrained frameworks for each possible symmetry group $\tau(\Gamma)$.
We say that $(G,\phi)$ is \emph{$\tau(\Gamma)$-tight} if it is tight, $\Gamma$-symmetric and satisfies the relevant constraints in Table~\ref{table:fixed}.

\begin{center}
\begin{table}[ht]
    \begin{tabular}{|c|c|}
    \hline
    $\tau(\Gamma)$ & $\text{Number of edges, loops and vertices fixed by symmetry operations}$ \\
    \hline
    $C_s$ & $e_{\sigma} + l_{\sigma,+} = l_{\sigma,-}$\\ 
    $C_n$, $n \neq 2,4$ & $v_n, e_n, l_n =0$\\
    $C_2$ & $v_2, e_2, l_2 =0$ or $v_2 = 1, e_2 = 0, l_2 =2$\\ 
    $C_4$ & $v_2, v_4, e_2, e_4, l_2, l_4 =0$ or $v_2, v_4 = 1, e_2, e_4 =0, l_2 =2, l_4 =0$\\ 
    $C_{2v}$ & $e_{\sigma} + l_{\sigma,+} = l_{\sigma,-}$, $(v_2, e_2, l_2 =0$ or $v_2 = 1, e_2 = 0, l_2 =2)$\\
    $C_{4v}$ & $e_{\sigma} + l_{\sigma,+} = l_{\sigma,-}$, $(v_2, v_4, e_2, e_4, l_2, l_4 =0$ or $v_2, v_4 = 1, e_2, e_4 =0, l_2 =2, l_4 =0)$\\
    $C_{nv}$, $n \neq 2,4$ & $e_{\sigma} + l_{\sigma,+} = l_{\sigma,-}$, $v_n, e_n, l_n =0$\\
    \hline
    \end{tabular}
    \caption{Fixed edge/loop/vertex counts for symmetry operations in the plane.}
    \label{table:fixed}
    \end{table}
\end{center}

\begin{thm}\label{thm: fixed 2dim lin}
Let $(G,p,q)$ be an isostatic $\tau(\Gamma)$-symmetric framework on the plane. Then $G$ is $\tau(\Gamma)$-tight.
\end{thm}

\begin{proof} 
Clearly $G$ is $\Gamma$-symmetric and tight by Theorem \ref{thm:lm}. That the constraints in Table~\ref{table:fixed} are satisfied follows immediately from Corollary \ref{lem: element count 2-dim} for the groups $C_s$ and $C_n$.
For dihedral groups we check if any of the constraints from Corollary \ref{lem: element count 2-dim} are mutually exclusive.
A vertex fixed by $c_2$ will be a point at the origin, and hence lie on any mirror.
The two $c_2$-fixed loops incident to this vertex are either images of each other under $\sigma$ or they are both fixed, i.e. for the corresponding linear constraints, one must lie on the mirror and the other perpendicular to the mirror.
With an additional $c_4$ symmetry operation, these linear constraints must be perpendicular.
Either of the two cases above is still possible, with the mirrors bisecting the angle between the linear constraints or one line lying on each mirror.
For $C_{nv}$ there will be no point at the origin.
Any vertices, edges or loops fixed by a mirror will have an orbit of $c_n$ symmetric copies, all fixed by their own respective mirrors.
None of this contradicts the $c_n$ counts.
\end{proof}

In the remainder of the article we consider the converse problem. Given a $\tau(\Gamma)$-tight graph and a $\tau(\Gamma)$-generic realisation, is it isostatic in the plane? For this combinatorial problem, we restrict attention to cyclic groups.

\section{Rigidity preserving operations}\label{sec:ops}

Given a $\tau(\Gamma)$-symmetric isostatic linearly constrained framework in $\mathbb{R}^2$, we next introduce several construction operations and prove that their application results in larger $\tau(\Gamma)$-symmetric isostatic linearly constrained frameworks in $\mathbb{R}^2$. Moreover we use this to show that a certain infinite family of $C_n$-symmetric linearly constrained frameworks are isostatic.
These construction operations are symmetry-adapted looped Henneberg-type graph operations. The operations are depicted in Figures \ref{fig:0ext} and \ref{fig: 1 ext} for specific symmetry groups.

We will work with an arbitrary finite group $\Gamma = \{\textrm{id} = \gamma_{0}, \gamma_{1}, \dots, \gamma_{t-1}\}$ and we will write $\gamma_{k}v$ instead of $\phi(\gamma_{k})(v)$ and often $\gamma_{k}(x,y)$ or $(x^{(k)},y^{(k)})$ for $\tau(\gamma_{k})(p(v))$ where $p(v) = (x,y)$.
For a group of order two, it will be common to write $v' = \gamma v$ for $\gamma \in \Gamma \setminus\{\textrm{id}\}$.

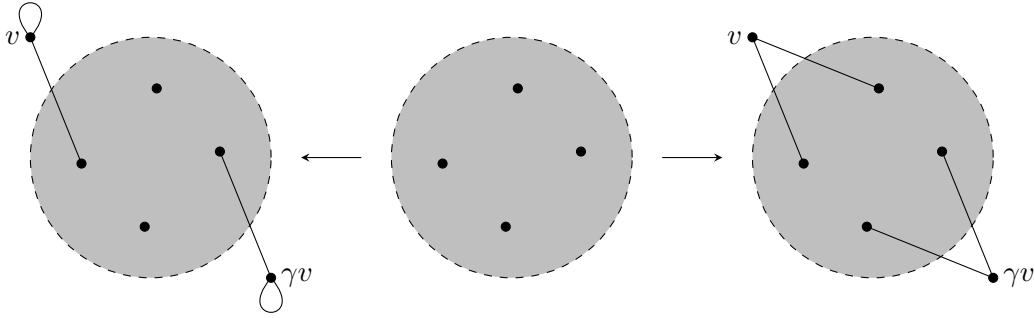
\begin{figure}[ht]
    \begin{center}
    \begin{tikzpicture}
  [scale=.4]
  
  \coordinate (n1) at (0,8);
  \coordinate (n2) at (1.7,3.8);
  \coordinate (n3) at (2.5,5.5);
  \coordinate (n4) at (4.2,6.3);
  \coordinate (n5) at (8,0);
  \coordinate (n6) at (3.8,1.7);
  \coordinate (n7) at (5.5,2.5);
  \coordinate (n8) at (6.3,4.2);
  \coordinate (n9) at (4,4);
  
 \draw[dashed, fill=lightgray] (n9) circle (4cm); 
 \draw[fill=black] (n2) circle (0.15cm);
 \draw[fill=black] (n4) circle (0.15cm);
 \draw[fill=black] (n6) circle (0.15cm);
 \draw[fill=black] (n8) circle (0.15cm);

  \coordinate (n11) at (-12,8);
  \coordinate (n12) at (-10.3,3.8);
  \coordinate (n13) at (-9.5,5.5);
  \coordinate (n14) at (-7.8,6.3);
  \coordinate (n15) at (-4,0);
  \coordinate (n16) at (-8.2,1.7);
  \coordinate (n17) at (-6.5,2.5);
  \coordinate (n18) at (-5.7,4.2);
  \coordinate (n19) at (-8,4);
  
 \draw[dashed, fill=lightgray] (n19) circle (4cm); 
 \draw[fill=black] (n11) circle (0.15cm)
	    node[left] {$v$};
 \draw[fill=black] (n12) circle (0.15cm);
 \draw[fill=black] (n14) circle (0.15cm);
 \draw[fill=black] (n15) circle (0.15cm)
	    node[right] {$\gamma v$};
 \draw[fill=black] (n16) circle (0.15cm);
 \draw[fill=black] (n18) circle (0.15cm);
 
   \foreach \from/\to in {n11/n12,n15/n18} 
    \draw (\from) -- (\to);

\draw[scale = 4] (-3,2)  to[in=50,out=130,loop] (0,0);
\draw[scale = 4] (-1,0)  to[in=-50,out=-130,loop] (0,0);
 
  \coordinate (n21) at (12,8);
  \coordinate (n22) at (13.7,3.8);
  \coordinate (n23) at (14.5,5.5);
  \coordinate (n24) at (16.2,6.3);
  \coordinate (n25) at (20,0);
  \coordinate (n26) at (15.8,1.7);
  \coordinate (n27) at (17.5,2.5);
  \coordinate (n28) at (18.3,4.2);
  \coordinate (n29) at (16,4);
  
 \draw[dashed, fill=lightgray] (n29) circle (4cm); 
 \draw[fill=black] (n21) circle (0.15cm)
	    node[left] {$v$};
 \draw[fill=black] (n22) circle (0.15cm);
 \draw[fill=black] (n24) circle (0.15cm);
 \draw[fill=black] (n25) circle (0.15cm)
	    node[right] {$\gamma v$};
 \draw[fill=black] (n26) circle (0.15cm);
 \draw[fill=black] (n28) circle (0.15cm);
	 
   \foreach \from/\to in {n21/n22,n21/n24,n25/n26,n25/n28} 
    \draw (\from) -- (\to);

\draw[-stealth] (9,4) -- (11,4);
\draw[-stealth] (-1,4) -- (-3,4);

\end{tikzpicture}
    \caption{$C_n$-symmetric 0-extensions adding new vertices $v$ and $\gamma v$ in each case; $n=2$ shown.}
    \label{fig:0ext}
    \end{center}
\end{figure}

\begin{figure}[ht]
    \begin{center}
    \begin{tikzpicture}
  [scale=.4]
  
  \coordinate (n100) at (4,4);
  \coordinate (n101) at (6.5,4);
  \coordinate (n102) at (6.165,5.25);
  \coordinate (n103) at (5.25,6.165);
  \coordinate (n104) at (2.75,6.165);
  \coordinate (n105) at (1.835,5.25);
  \coordinate (n106) at (1.5,4);
  \coordinate (n107) at (2.75,1.835);
  \coordinate (n108) at (4,1.5);
  \coordinate (n109) at (5.25,1.835);
  
 \draw[dashed, fill=lightgray] (n100) circle (4cm); 
 \draw[fill=black] (n101) circle (0.15cm);
 \draw[fill=black] (n102) circle (0.15cm);
 \draw[fill=black] (n103) circle (0.15cm);
 \draw[fill=black] (n104) circle (0.15cm);
 \draw[fill=black] (n105) circle (0.15cm);
 \draw[fill=black] (n106) circle (0.15cm);
 \draw[fill=black] (n107) circle (0.15cm);
 \draw[fill=black] (n108) circle (0.15cm);
 \draw[fill=black] (n109) circle (0.15cm);
 
\draw[scale = 4] (1.625,1)  to[in=-40,out=40,loop] (0,0);
\draw[scale = 4] (2.75/4,6.165/4)  to[in=80,out=160,loop] (0,0);
\draw[scale = 4] (2.75/4,1.835/4)  to[in=-160,out=-80,loop] (0,0);
	    
  \coordinate (n0) at (16,4);
  \coordinate (n1) at (18.5,4);
  \coordinate (n2) at (18.165,5.25);
  \coordinate (n3) at (17.25,6.165);
  \coordinate (n4) at (14.75,6.165);
  \coordinate (n5) at (13.835,5.25);
  \coordinate (n6) at (13.5,4);
  \coordinate (n7) at (14.75,1.835);
  \coordinate (n8) at (16,1.5);
  \coordinate (n9) at (17.25,1.835);
  \coordinate (n10) at (20.899,6.828);
  \coordinate (n11) at (11.101,6.828);
  \coordinate (n12) at (16,-1.657);
  
 \draw[dashed, fill=lightgray] (n0) circle (4cm); 
 \draw[fill=black] (n1) circle (0.15cm);
 \draw[fill=black] (n2) circle (0.15cm);
 \draw[fill=black] (n3) circle (0.15cm);
 \draw[fill=black] (n4) circle (0.15cm);
 \draw[fill=black] (n5) circle (0.15cm);
 \draw[fill=black] (n6) circle (0.15cm);
 \draw[fill=black] (n7) circle (0.15cm);
 \draw[fill=black] (n8) circle (0.15cm);
 \draw[fill=black] (n9) circle (0.15cm);
 \draw[fill=black] (n10) circle (0.15cm)
	    node[below right] {$v$};
 \draw[fill=black] (n11) circle (0.15cm)
	    node[below left] {$\gamma v$};
 \draw[fill=black] (n12) circle (0.15cm)
	    node[left] {$\gamma^{2} v$};
 
\draw[scale = 4] (8.899/4+3,6.828/4)  to[in=-10,out=70,loop] (0,0);
\draw[scale = 4] (-12.899/4+6,6.828/4)  to[in=110,out=190,loop] (0,0);
\draw[scale = 4] (4,-1.657/4)  to[in=-130,out=-50,loop] (0,0);

   \foreach \from/\to in {n1/n10,n3/n10,n4/n11,n6/n11,n7/n12,n9/n12} 
    \draw (\from) -- (\to);
 
  \coordinate (n200) at (16,16);
  \coordinate (n201) at (18.5,16);
  \coordinate (n202) at (18.165,17.25);
  \coordinate (n203) at (17.25,18.165);
  \coordinate (n204) at (14.75,18.165);
  \coordinate (n205) at (13.835,17.25);
  \coordinate (n206) at (13.5,16);
  \coordinate (n207) at (14.75,13.835);
  \coordinate (n208) at (16,13.5);
  \coordinate (n209) at (17.25,13.835);
  \coordinate (n210) at (20.899,18.828);
  \coordinate (n211) at (-12.899+24,18.828);
  \coordinate (n212) at (16,-1.657+12);

 \draw[dashed, fill=lightgray] (n200) circle (4cm); 
 \draw[fill=black] (n201) circle (0.15cm);
 \draw[fill=black] (n202) circle (0.15cm);
 \draw[fill=black] (n203) circle (0.15cm);
 \draw[fill=black] (n204) circle (0.15cm);
 \draw[fill=black] (n205) circle (0.15cm);
 \draw[fill=black] (n206) circle (0.15cm);
 \draw[fill=black] (n207) circle (0.15cm);
 \draw[fill=black] (n208) circle (0.15cm);
 \draw[fill=black] (n209) circle (0.15cm);
 \draw[fill=black] (n210) circle (0.15cm)
	    node[right] {$v$};
 \draw[fill=black] (n211) circle (0.15cm)
	    node[left] {$\gamma v$};
 \draw[fill=black] (n212) circle (0.15cm)
	    node[left] {$\gamma^{2} v$};

   \foreach \from/\to in {n201/n210,n202/n210,n203/n210,n204/n211,n205/n211,n206/n211,n207/n212,n208/n212,n209/n212} 
    \draw (\from) -- (\to);
    
\draw[-stealth] (9,4) -- (11,4);
\draw[-stealth] (9,16) -- (11,16);

  \coordinate (n300) at (4,16);
  \coordinate (n301) at (6.5,16);
  \coordinate (n302) at (6.165,17.25);
  \coordinate (n303) at (5.25,18.165);
  \coordinate (n304) at (2.75,18.165);
  \coordinate (n305) at (1.835,17.25);
  \coordinate (n306) at (1.5,16);
  \coordinate (n307) at (2.75,13.835);
  \coordinate (n308) at (4,13.5);
  \coordinate (n309) at (5.25,13.835);
  
 \draw[dashed, fill=lightgray] (n300) circle (4cm); 
 \draw[fill=black] (n301) circle (0.15cm);
 \draw[fill=black] (n302) circle (0.15cm);
 \draw[fill=black] (n303) circle (0.15cm);
 \draw[fill=black] (n304) circle (0.15cm);
 \draw[fill=black] (n305) circle (0.15cm);
 \draw[fill=black] (n306) circle (0.15cm);
 \draw[fill=black] (n307) circle (0.15cm);
 \draw[fill=black] (n308) circle (0.15cm);
 \draw[fill=black] (n309) circle (0.15cm);

   \foreach \from/\to in {n301/n302,n304/n305,n307/n308} 
    \draw (\from) -- (\to);
	    
\end{tikzpicture}
    \caption{$C_n$-symmetric 1-extensions adding new vertices $v$, $\gamma v$ and $\gamma^{2} v$ in each case; $n=3$ shown.}
    \label{fig: 1 ext}
    \end{center}
\end{figure}
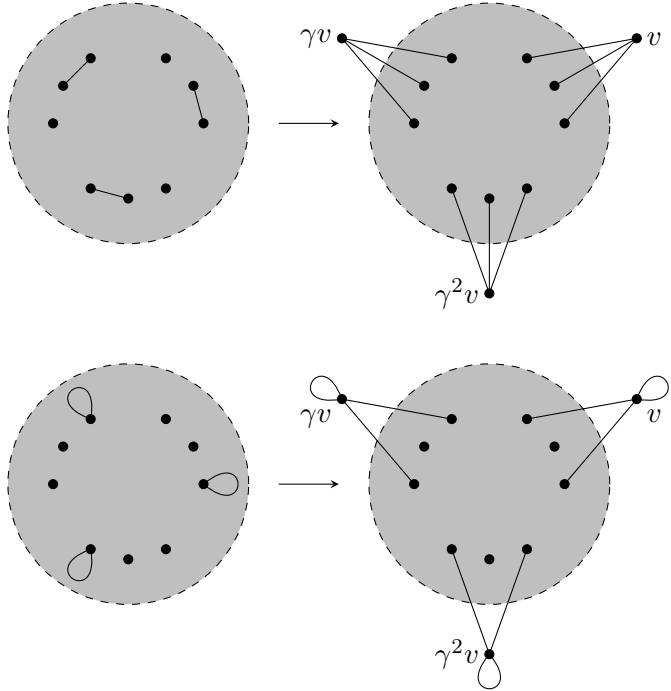

Let $G=(V,E, L)$ be a $\Gamma$-symmetric looped simple graph for a group $\Gamma$ of order $t$.
Then a \emph{symmetrised 0-extension} creates a new $\Gamma$-symmetric looped simple graph $G^+=(V^+,E^+, L^+)$ by adding the $t$ vertices $\{v, \gamma v, \dots, \gamma_{t-1}v\}$ with either: $v$ adjacent to $v_{i},v_{j}$, and for each $k \in \{1,\dots,t-1\}$, $\gamma_{k}v$ adjacent to $\gamma_{k}v_{i},\gamma_{k}v_{j}$; or $v$ adjacent to $v_{i}$ and incident to the loop $(v,v)$, and for each $k \in \{1,\dots,t-1\}$, $\gamma_{k}v$ adjacent to $\gamma_{k}v_{i}$ and incident to $(\gamma_k v_i, \gamma_k v_i)$.

Let $e_i=x_iy_i \in E$, $i=0\leq i \leq t-1$ be an edge orbit of $G$ of size $t$ under the action of $\Gamma$.
Further let $z_0\neq x_0,y_0$ and let $z_i=\gamma_i z_0$ for $i=1,\ldots, t-1$.
A \emph{symmetrised 1-extension} creates a new $\Gamma$-symmetric looped simple graph by adding $t$ vertices $\{v, \gamma v, \dots, \gamma_{t-1}v\}$ and deleting all the edges $e_i$ from $G$, and  with $v$ adjacent to $x_0,y_0$ and $z_0$, and $\gamma_iv$ adjacent to $x_{i},y_{i}$ and $z_{i}$ for $i=1,\ldots, t-1$.
Alternatively, let $l_i = x_i x_i \in L$ for $i=0\leq i \leq t-1$ be a loop orbit of $G$ of size $t$ under the action of $\Gamma$.
A \emph{symmetrised looped 1-extension} creates a new $\Gamma$-symmetric looped simple graph by adding $t$ vertices $\{v, \gamma v, \dots, \gamma_{t-1}v\}$ and $t$ loops $l_i^*= (\gamma _i v, \gamma_i v)$ for $i=0\leq i \leq t-1$ and deleting all the loops $l_i$ from $G$, and with $v$ adjacent to $x_0$ and $y_0$, and $\gamma_iv$ adjacent to $x_{i}$ and $y_{i}$ for $i=1,\ldots, t-1$.


\begin{lem}\label{0-ext rigid}
Suppose $G$ is $\Gamma$-symmetric. Let $G^{+}$ be obtained from $G$ by a symmetrised $0$-extension. If $(G,p,q)$ is $\tau(\Gamma)$-isostatic in $\mathbb{R}^2$, then for appropriate maps $p^+, q^+$, $(G^{+},p^+, q^+)$ is $\tau(\Gamma)$-isostatic in $\mathbb{R}^2$.
\end{lem}

\begin{proof}
There are two cases to consider here for the two variants of 0-extensions.
We first consider a new orbit of vertices adjacent to two vertices.
Write $G^{+} = G + \{v,\dots, \gamma_{t-1}v\}$, and let $v \in V^{+}$ be adjacent to $v_{1},v_{2}$, and for each $k \in \{1,\dots,t-1\}$, $\gamma_{k}v$ adjacent to $\gamma_{k}v_{1},\gamma_{k}v_{2}$.
Define $p^{+}: V^{+} \to \R^{2}$ by $p^{+}(z)=p(z)$ for all $z\in V$, $p^{+}(v) = (x,y)$, and $p^{+}(\gamma_{k}v) = (x^{(k)},y^{(k)})$. Write $p(v_{1}) = (x_{1},y_{1})$, $p(v_{2}) = (x_{2},y_{2})$. Then,
\[ R(G^{+},p^{+},q) =
\left[ \begin{tabular}{ c c c c c c c }
    $R(G,p,q)$ & & & & & & \\
    & $x - x_{1}$ & $y - y_{1}$ & & & &  \\
    *& $x - x_{2}$ & $y - y_{2}$ & & $\mathbf{0}$ & & \\
    & & & $\ddots$ & & & \\
    & & & &  $x^{(k)} - x_{1}^{(k)}$ & $y^{(k)} - y_{1}^{(k)}$ & \\
    * & $\mathbf{0}$ & & & $x^{(k)} - x_{2}^{(k)}$ & $y^{(k)} - y_{2}^{(k)}$ & \\
    & & & & & & $\ddots$
    \end{tabular} \right], \]
and hence the fact that  $R(G^{+},p^{+},q)$ has linearly independent rows will follow once each $2 \times 2$ submatrix indicated above is shown to be invertible.
Choose $(x,y)$ so that $p^+(v), p^+(v_1), p^+(v_2)$ are not collinear, that is $(x,y) \neq (x_1,y_1)+\lambda(x_2-x_1,y_2-y_1)$ for any $\lambda \in \R$, which is exactly the requirement for the first submatrix to be invertible.
Since each $\tau(\gamma_{k})$ is an isometry, all of the other $t-1$ remaining submatrices are also invertible, and so $\textrm{rank } R(G^{+},p^{+},q) \geq \textrm{rank }R(G,p,q) + 2t$.
Hence, if $G$ is $\tau(\Gamma)$-independent so is $G^{+}$.

Alternatively, consider a new orbit of vertices with one loop incident, adjacent to one vertex.
Write $G^+= G + \{v,\dots, \gamma_{t-1}v\}$ and let $v \in V^{+}$ be adjacent to $v_{1}$ and incident to the loop $l' = vv$, and for each $k \in \{1,\dots,t-1\}$, $\gamma_{k}v$ adjacent to $\gamma_{k}v_{1}$ and incident to $\gamma_k l$.
Define $p^+: V^+ \to \R^{2}$ by $p^+(z)=p(z)$ for all $z\in V$, $p^+(v) = (x,y)$, and $p^+(\gamma_{k}v) = (x^{(k)},y^{(k)})$ and $q^+: L^+ \to \R^{2}$ by $q^+=q(l)$ for all $l\in L$, $q^+(l') = (u,w)$, and $q^+(\gamma_{k}l') = (u^{(k)},w^{(k)})$. Write $p(v_{1}) = (x_{1},y_{1})$. Then,
\[ R(G^+,p^+,q^+) =
\left[ \begin{tabular}{ c c c c c c c }
    $R(G,p,q)$ & & & & & & \\
    & $x - x_{1}$ & $y - y_{1}$ & & & &  \\
    *& $u$ & $w$ & & $\mathbf{0}$ & & \\
    & & & $\ddots$ & & & \\
    & & & &  $x^{(k)} - x_{1}^{(k)}$ & $y^{(k)} - y_{1}^{(k)}$ & \\
    * & $\mathbf{0}$ & & & $u^{(k)}$ & $w^{(k)}$ & \\
    & & & & & & $\ddots$
    \end{tabular} \right], \]
and hence the fact that  $R(G^+,p^+,q^+)$ has linearly independent rows will follow once each $2 \times 2$ submatrix indicated above is shown to be invertible.
Choose $p^+(v), q^+(l')$ so that $(u,w) \neq \lambda(x-x_1,y-y_1)$ for any $\lambda \in \R$ (equality in the equation corresponds to the linear constraint being perpendicular to the edge $vv_1$), hence the first submatrix is invertible.
Since each $\tau(\gamma_{k})$ is an isometry, all of the other $t-1$ remaining submatrices are also invertible, and so $\textrm{rank } R(G^+,p^+,q^+) \geq \textrm{rank }R(G,p,q) + 2t$.
Hence, if $G$ is $\tau(\Gamma)$-independent so is $G^+$.
\end{proof}

\begin{rem}\label{rem: Cs fixed 0-ext}
For a $C_s$-symmetric graph $G$, let $G^+ = G + {v}$ be obtained by a $0$-extension with a vertex $v$ fixed by $\sigma$.
Then by letting $y_j = -y_i$, the relevant $2 \times 2$ matrix (see the proof above) is still invertible, hence if $G$ is $C_s$-independent, then $G^{+}$ is $C_s$-independent.
\end{rem}

\begin{lem}\label{1-ext rigid}
Let $G$ be a $\Gamma$-symmetric graph, and $G^{+}$ be obtained from $G$ by a symmetrised $1$-extension.
If $G$ is $\tau(\Gamma)$-isostatic, then $G^{+}$ is $\tau(\Gamma)$-isostatic.
\end{lem}

\begin{proof}
Suppose first that the 1-extension adds a new set of vertices, $\{v, \gamma v, \dots, \gamma^{t-1} v\}$, where for each $i \in \{0, \dots, t-1\}$, $\gamma^i v$ is adjacent to three vertices, say $\gamma^i v_1, \gamma^i v_2, \gamma^i v_3$, and the edges $(\gamma^i v_1, \gamma^i v_2)$ are deleted.
Call this new graph $G^+$.
Write $p(v_{1}) = (x_{1},y_{1})$, $p(v_{2}) = (x_{2},y_{2})$, $p(v_{3}) = (x_{3},y_{3})$.
Define $p^{+}: V^{+} \to \R^{2}$ by putting $p^{+}(z)=p(z)$ for all $z\in V$, and choosing special positions so that $p^{+}(v) =  \frac{1}{2}(p(v_1)+p(v_2)) = (x,y)$, and $p^{+}(\gamma_{k}v) = (x^{(k)},y^{(k)})$. 
Now consider the rigidity matrix of the realisation of $K_3$ with vertex positions $p^+(v), p^+(v_1), p^+(v_2)$.
Then,
\begin{align*}
R(K_3,p^{+}) = & \left[ \begin{tabular}{ c c c }
    $p^+(v_1) - p^+(v_2)$ & $p^+(v_2) - p^+(v_1)$ & $\mathbf{0}$ \\
    $p^+(v_1) - p^+(v)$ & $\mathbf{0}$ & $p^+(v) - p^+(v_1)$ \\
    $\mathbf{0}$ & $(p^+(v_2) - p^+(v)$ & $p^+(v) - p^+(v_2)$ \\
    \end{tabular} \right]\\
= & \left[ \begin{tabular}{ c c c }
    $p^+(v_1) - p^+(v_2)$ & $p^+(v_2) - p^+(v_1)$ & $\mathbf{0}$ \\
    $\frac{1}{2}(p^+(v_1) - p^+(v_2))$ & $\mathbf{0}$ & $\frac{1}{2}(p^+(v_2) - p^+(v_1))$ \\
    $\mathbf{0}$ & $\frac{1}{2}(p^+(v_2) - p^+(v_1))$ & $\frac{1}{2}(p^+(v_1) - p^+(v_2))$ \\
    \end{tabular} \right]
\end{align*}
has rank 2 and the linear dependence is non-zero on all 3 rows.
We note that the $\tau(\gamma)$ preserves this linear dependence.
Then, since $(G^+ + \bigcup_{i=0}^{t-1}\{(\gamma^i v_1, \gamma^i v_2)\} \setminus \bigcup_{i=0}^{t-1}\{(\gamma^i v, \gamma^i v_2)\}, p^+, q)$ is obtained from $(G,p,q)$ by a symmetrised 0-extension, $$\rank R(G^+ + \bigcup_{i=0}^{t-1}\{(\gamma^i v_1, \gamma^i v_2)\} \setminus \bigcup_{i=0}^{t-1}\{(\gamma^i v, \gamma^i v_2)\}, p^+, q) = \rank R(G,p,q) + 2t.$$
We observe from $R(K_3,p^{+})$ that $\rank R(G^+ + \bigcup_{i=0}^{t-1}\{(\gamma^i v_1, \gamma^i v_2)\}, p^+, q) = \rank R(G,p,q) + 2t$ too, and further, since any row can be written as a linear combination of the other two, we can delete any edge orbit from the orbit of the $K_3$ subgraphs and preserve infinitesimal rigidity.
Since $(G^+,p^+,q)$ is infinitesimally rigid in this special position, $G^+$ is $\tau(\Gamma)$-isostatic.

Consider now a looped 1-extension that creates a new graph $G^+$ from $G$ by adding vertices $\{v, \gamma v, \dots, \gamma^{t-1} v\}$, where for each $i \in \{0, \dots, t-1\}$, $\gamma^i v$ is adjacent to two vertices, say $\gamma^i v_1, \gamma^i v_2$, and incident to the new loop $\gamma^i l'$, with the loops $\gamma^ i l = (\gamma^i v_1, \gamma^i v_1)$ deleted.
Write $p(v_{1}) = (x_{1},y_{1})$, $p(v_{2}) = (x_{2},y_{2})$.
Define $p^{+}: V^{+} \to \R^{2}$ by putting $p^{+}(z)=p(z)$ for all $z\in V$, and choosing special positions $p^{+}(v) = p(v_1)+q(l) = (x,y)$, and $p^{+}(\gamma_{k}v) = (x^{(k)},y^{(k)})$.
Further define $q^{+}: L^{+} \to \R^2$ by $q^{+}(h)=q(h)$ for all $h \in L$  and the special position $q^{+}(l') = q^{+}(l) = (c,d)$ and symmetrically the loops $\{\gamma l', \dots, \gamma^{t-1}l'\}$.
We then consider the rigidity matrix for the realisation of $H = G[\{v,v_1\}]$ with the vertex positions $p^+(v), p^+(v_1)$ and linear constraints $q^+(l), q^+(l')$.
Then,
\[
R(H,p^{+},q^{+}) = \left[ \begin{tabular}{ c c }
    $p^+(v_1) - p^+(v)$ & $p^+(v) - p^+(v_1)$ \\
    $q^+(l)$ & $\mathbf{0}$\\
    $\mathbf{0}$ & $q^+(l')$ \\
    \end{tabular} \right]\\
= \left[ \begin{tabular}{ c c }
    $(-c,-d)$ & $(c,d)$ \\
    $(c,d)$ & $\mathbf{0}$\\
    $\mathbf{0}$ & $(c,d)$ \\
    \end{tabular} \right]
\]
has rank 2 and the linear dependence is non-zero on all 3 rows.
We note that $\tau(\gamma)$ preserves this linear dependence.
Then since $(G^+ + \bigcup_{i=0}^{t-1}\{(\gamma^i v_1, \gamma^i v_1)\} \setminus \bigcup_{i=0}^{t-1}\{(\gamma^i v, \gamma^i v_1)\}, p^+, q^+)$ is obtained from $(G,p,q)$ by a symmetrised 0-extension, $$\rank R(G^+ + \bigcup_{i=0}^{t-1}\{(\gamma^i v_1, \gamma^i v_1)\} \setminus \bigcup_{i=0}^{t-1}\{(\gamma^i v, \gamma^i v_1)\}, p^+, q^+) = \rank R(G,p,q) + 2t.$$
We then see from $R(H,p^{+},q^{+})$ that $\rank R(G^+ + \bigcup_{i=0}^{t-1}\{(\gamma^i v_1, \gamma^i v_1)\}, p^+, q) = \rank R(G,p,q) + 2t$ too, and further since any row can be written as a linear combination of the other two, we can delete one orbit of $\{l,l',vv_1\}$ and preserve infinitesimal rigidity.
Since $(G^+,p^+,q^+)$ is infinitesimally rigid in this special position, $G^+$ is $\tau(\Gamma)$-isostatic.
\end{proof}

We next show some specific families of graphs are $\Gamma$-isostatic for certain groups. In the next sections these families will turn out to be base graphs for our construction arguments.
A \emph{pinned graph} on $n$ vertices, is the graph with 2 loops incident to each vertex and $E = \emptyset$, which we denote $\mathcal{P}_n$.
A \emph{looped $n$-cycle} is a cycle on $n$ vertices with a single loop incident to each vertex (see Figure \ref{fig: c_5 base}), which we denote $LC_n$.
Define the symmetric graph $(\mathcal{P}_1,\phi_0)$ to have $\phi_0: \Z/{2\Z} \to \Aut(\mathcal{P}_1)$ fix all elements of $\mathcal{P}_1$ for all $\gamma \in \Z/{2\Z}$.
Furthermore, $(\mathcal{P}_1,\phi_1)$ to have $\phi_1: \Z/{2\Z} \to \Aut(\mathcal{P}_1)$ fix the vertex and transpose the loops for the non-trivial element of $\Z/{2\Z}$.
Let $(\mathcal{P}_n,\phi_n)$ be defined by $\phi_n: \Z/{n\Z} \to \Aut(\mathcal{P}_n)$ to have a single orbit of the $n$ vertices.
Lastly, $(LC_n,\psi_n)$ has $\psi_n: \Z/{n\Z} \to \Aut(LC_n)$ form a single orbit of the $n$ vertices.

\begin{lem}\label{lem: base rigid}
    The following graphs or graph classes are $\tau(\Gamma)$-isostatic:
    \begin{itemize}
        \item $(\mathcal{P}_1,\phi_0)$ is $C_2$-isostatic;
        \item $(\mathcal{P}_1,\phi_1)$ is $C_4$-isostatic;
        \item for $n \geq 2$, $(\mathcal{P}_n,\phi_n)$ is $C_n$-isostatic;
        \item for odd $n\geq 3$, $(LC_n,\psi_n)$ is $C_n$-isostatic. 
    \end{itemize}
\end{lem}

\begin{proof}
    In the first two bullet points, we have a single vertex restricted by two linear constraints.
    This will be pinned unless the corresponding loops are $c_2$ images of each other, in which case the linear constraints must coincide.
    In the third bullet point, every vertex is pinned.
    Therefore in the first three bullet points the graphs are $\tau(\Gamma)$-isostatic.
    In the final bullet point, we put the framework in special position, so that the linear constraints pass through the origin (that is the vertices can only move radially).
    Any infinitesimal motion of a vertex can therefore be described as ``inward" or ``outward" from the origin, depending on whether the radial distance decreases or increases.
    In order for the edge lengths to be preserved, any inward moving vertex must be adjacent to two outward moving vertices, and likewise any outward moving vertex must be adjacent to two inward moving vertices.
    As a result of this, inward and outward moving vertices must alternate in the cycle.
    Thus there are as many outward as inward moving vertices, which is not possible with $n$ odd.
    Hence, this special position is infinitesimally rigid, and so the graph is $\tau(\Gamma)$-isostatic.
\end{proof}

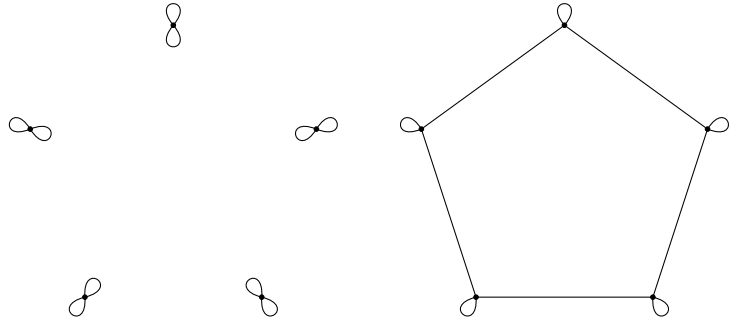
\begin{figure}[ht]
    \centering
    \begin{tikzpicture}
  [scale=0.2,auto=left]

  \coordinate (n6) at (0,10);
  \coordinate (n7) at (9.511,3.090);
  \coordinate (n8) at (5.878,-8.090);
  \coordinate (n9) at (-5.878,-8.090);
  \coordinate (n10) at (-9.511,3.090);

 \draw[fill=black] (n6) circle (0.15cm);
 \draw[fill=black] (n7) circle (0.15cm);
 \draw[fill=black] (n8) circle (0.15cm);
 \draw[fill=black] (n9) circle (0.15cm);
 \draw[fill=black] (n10) circle (0.15cm);

\draw[scale = 5] (0,2)  to[in=50,out=130,loop] (0,0);
\draw[scale = 5] (0,2)  to[in=-50,out=-130,loop] (0,0);
\draw[scale = 5] (1.9022,0.618)  to[in=-22,out=58,loop] (0,0);
\draw[scale = 5] (1.9022,0.618)  to[in=158,out=-122,loop] (0,0);
\draw[scale = 5] (1.1756,-1.618)  to[in=86,out=166,loop] (0,0);
\draw[scale = 5] (1.1756,-1.618)  to[in=-94,out=-14,loop] (0,0);
\draw[scale = 5] (-1.1756,-1.618)  to[in=-166,out=-86,loop] (0,0);
\draw[scale = 5] (-1.1756,-1.618)  to[in=14,out=94,loop] (0,0);
\draw[scale = 5] (-1.9022,0.618)  to[in=122,out=202,loop] (0,0);
\draw[scale = 5] (-1.9022,0.618)  to[in=302,out=22,loop] (0,0);

  \coordinate (n16) at (26,10);
  \coordinate (n17) at (35.511,3.090);
  \coordinate (n18) at (31.878,-8.090);
  \coordinate (n19) at (20.122,-8.090);
  \coordinate (n20) at (16.489,3.090);

 \draw[fill=black] (n16) circle (0.15cm);
 \draw[fill=black] (n17) circle (0.15cm);
 \draw[fill=black] (n18) circle (0.15cm);
 \draw[fill=black] (n19) circle (0.15cm);
 \draw[fill=black] (n20) circle (0.15cm);

  \foreach \from/\to in {n16/n17,n17/n18,n18/n19,n19/n20,n20/n16} 
    \draw (\from) -- (\to);
\draw[scale = 5] (5.2,2)  to[in=50,out=130,loop] (0,0);
\draw[scale = 5] (7.1022,0.618)  to[in=-22,out=58,loop] (0,0);
\draw[scale = 5] (6.3756,-1.618)  to[in=-94,out=-14,loop] (0,0);
\draw[scale = 5] (4.0244,-1.618)  to[in=-166,out=-86,loop] (0,0);
\draw[scale = 5] (3.2978,0.618)  to[in=122,out=202,loop] (0,0);

\end{tikzpicture}
    \caption{Base graphs for the construction of $C_n$-symmetric ($n\geq 3$) linearly constrained isostatic frameworks; $(\mathcal{P}_5, \phi_5)$ and $(LC_5,\psi_5)$ depicted.}
    \label{fig: c_5 base}
\end{figure}

Using another special position argument, Lemma \ref{lem: base rigid} can be extended to show all $C_n$-symmetric looped $n$-cycles are $\tau(\Gamma)$-isostatic \cite{Wall}. In this paper, however, we will only focus on rotational groups of order 2 or of odd order.


\section{$C_2$-symmetric isostatic graphs}
\label{sec:C2}

In the following sections we prove our main results.
These results show that the standard sparsity counts, together with the necessary conditions derived in Section \ref{sec:necrep}, are also sufficient for generic symmetric frameworks to be isostatic.  The proofs are based on a recursive construction using the Henneberg-type moves discussed in the previous section. 
We first consider $C_2$ rotational symmetry.
Recall that a $C_2$-tight graph is a graph with $v_2 = e_2 =l_2 = 0$ or $v_2 =1, e_2 = 0, l_2 =2$. 

\subsection{Graph theoretic preliminaries}
We will use standard graph theoretic terminology. 
For a graph $G=(V,E)$, $\delta(G)$ will denote the minimum degree of $G$, and $N(v)$ will denote the neighbours of a vertex $v\in V$.
The degree of a vertex $v$ is denoted $d_G(v)$.
For $X\subset V$ we will use $i_{E+L}(X)$, $i_E(X)$, $i_L(X)$ to denote the number of edges and loops, edges, loops in the induced subgraph $G[X]$ respectively, and the set $X$ will be called \emph{k-critical} and \emph{k-edge-critical}, for $k\in \mathbb{N}$, if $i_{E+L}(X)=2|X|-k$, and $i_E(X)=2|X|-k$ respectively.
We also write $k_X$ and $\Bk_X$ to denote the critical and edge-critical values, that is  $k_X= 2|X|-i_{E+L}(X)$ and $\Bk_X = 2|X|-i_E(X)$ for $X \subset V$. 
For $X,Y\subset V$, $d_G(X,Y)$ will denote the number of edges of the form $xy$ with $x\in X\setminus Y$ and $y\in Y\setminus X$.
We will often suppress subscripts when the graph is clear from the context and use $d(v)$ and $d(X,Y)$.

We also say that a subset $X$ of $V$ is \emph{$\tau(\Gamma)$-symmetric} (\emph{sparse}) if $G[X]$ is $\tau(\Gamma)$-symmetric (sparse).
Note that if $X$ is $\tau(\Gamma)$-symmetric and for each $\gamma \in \Gamma\setminus\{\textrm{id}\}$, $\gamma$ has no fixed edges or loops, then $|\Gamma|$ divides $k_X$ and $\Bk_X$.
Additionally, we note that a $C_2$-tight graph with fixed edge, loop, and vertex counts described in Table \ref{table:fixed}, has either no fixed loops or two fixed loops incident to a fixed vertex.
Therefore any symmetric vertex set will induce a subgraph with no fixed loops or two fixed loops, and so $|C_2| = 2$ will still divide $k_X$.
This fact will be crucial in what follows.

\begin{lem}\label{critcapcup}
Let $G$ be a graph and suppose $A,B \subseteq V$ have non-empty intersection.
Then $k_A + k_B = k_{A \cup B} + k_{A \cap B} + d(A,B)$ and $\Bk_A + \Bk_B = \Bk_{A \cup B} + \Bk_{A \cap B} + d(A,B)$.
\end{lem}

\begin{proof}
Since $|A| + |B| = |A \cup B| + |A \cap B|$, we have
\begin{equation}\label{counting W}
    \begin{split}
        2|A|-k_{A} + 2|B|-k_{B} & =  i_{E+L}(A) + i_{E+L}(B) = i_{E+L}(A\cup B) + i_{E+L}(A \cap B) - d(A,B) \\
        & = 2|A \cup B| - k_{A \cup B} + 2|A \cap B| - k_{A \cap B} - d(A,B) \\
        & = 2|A| + 2|B| - k_{A \cup B} - k_{A \cap B} - d(A,B),
    \end{split}
\end{equation}
giving the result.
The edge-critical variant is identical.
\end{proof}

Most of the technical work in the next two sections involves analysing when we can remove a vertex of degree 3.
Hence, for brevity, we will call a vertex of degree 3 a \emph{node}.
We will consider \emph{reduction operations}: these are the reverse of the extension operations described in Section \ref{sec:ops}, namely symmetrised 0-reductions and symmetrised 1-reductions.

\subsection{Reduction operations}

\begin{lem}\label{lem: c_2 two edges redu}
Let $(G,\phi)$ be a $C_{2}$-tight graph containing a vertex $v$ incident to two edges.
Then $G \setminus \{v,v'\}$ is $C_2$-tight.
\end{lem}

Note that this lemma includes the cases when $v$ has degree two and is adjacent to two distinct vertices and when $v$ has degree three with a loop at $v$ (recall Figure \ref{fig:0ext}).

\begin{proof}
If either $G - v$ or $G \setminus \{v,v'\}$ breaks sparsity, $G$ would not be tight.
\end{proof}

In the proof of the following lemma, spefically the second paragraph of Case 2, we prove the following remark which we will reference in future proofs.
\begin{rem}\label{rem: no 0c and 4e}
Suppose $(G,\phi)$ is $C_2$-tight containing a node $v\in V$ with $N(v)=\{x,y,z\}$. Then it is impossible for  $x,y$ to be in a 0-critical set and $x,z,x',z'$ to be in a $4$-edge-critical set. This conclusion is independent of the number of edges induced by $N(v)$.
\end{rem}

\begin{lem}\label{lem: c2 0,1,3,4 crit}
Let $(G,\phi)$ be a $C_{2}$-tight graph and suppose $v \in V$ is a node with $N(v) =\{x,y,z\}$ and $N(v) \cap N(v') = \emptyset$.
For a pair $x_1,x_2 \in \{x,y,z\}$ with $x_1 x_2 \notin E$, the following holds:
$x_1, x_2$ is not contained in any $0$-critical set or any $3$-edge-critical set, and
$x_1, x_2, x_1', x_2'$ is not contained in any $1$-critical set or any $4$-edge-critical set.
\end{lem}

\begin{proof}
We split up the proof based on the number of edges of $G$ induced by $N(v)$. 
If three edges were present $G[N(v)\cup\{v\}] \cong K_4$ which is not sparse, so we may assume $xy\notin E$.
Suppose the pair $x,y$ is not in a $0$-critical set and suppose there exists a $1$-critical set $W$ containing $x, y, x', y'$.
Then $W \cup W'$ and $W \cap W'$ are both $C_2$-symmetric so have even criticality since $G$ is $C_2$-tight.
By Lemma \ref{critcapcup} one must be $0$-critical, a contradiction since both contain $x, y$.
Hence we know that if $x, y \in N(v)$ are not in a $0$-critical set, then $x, y, x', y'$ are not in a $1$-critical set.

\emph{Case 1:} If there are two edges on the vertices of $N(v)$, then $x,y$ cannot be contained in a $3$-edge-critical subset and $x,y,x',y'$ are not in a $4$-edge critical subset as these sets with $v, z$ and $v, z, v', z'$ added respectively would not be sparse.
It's easy to see that there is no $0$-critical set on $x,y$, and there is no $1$-critical set on $x,y,x',y'$ from the paragraph above, finishing Case 1.

\emph{Case 2:} Suppose exactly one edge, say $xz$, is present on the vertices of $N(v)$.
First assume $x,y$ is contained in a $0$-critical set, $X$.  Then it is easy to check that $y, z$ is not in a $0$-critical set.
Assume there exists a $3$-edge-critical set $U$ on $y,z$.
This means there is $l_U\in \{0,1,2,3\}$ loops on the vertices of $U$, with $U$ being $(3-l_U)$-critical.
By Lemma \ref{critcapcup}, $$k_{X \cup U} = k_X + k_U - k_{X \cap U} - d(X,U).$$
We know $k_X = 0$, $X \cap U$ is at least $2$-edge-critical and $X \cap U \subset U$ so there are $l_{X \cap U}\leq l_U$ loops on the vertices of $X \cap U$.
Therefore, $$k_U - k_{X \cap U} = 3-l_U - \Bk_{X\cap U}+l_{X\cap U} \leq 1,$$ and $xz$ gives that $d(X,U)$ is at least 1, so $X\cup U$ is $0$-critical. But then $X\cup U +\{v\}$ is not sparse in $G$, a contradiction.

To show $y,z,y',z'$ is not contained in a $4$-edge-critical set, first notice that $X \cup X'$ is $0$-critical, containing $x,y,x',y'$. Relabel this union as $X$.
Assume there is a $4$-edge-critical set $W$ on $y,z,y',z'$.
We take the symmetric sets $W \cup W'$ and $W \cap W'$, which by symmetry have even edge-criticality.
Additionally, $y,z,y',z' \in W \cap W'$, so that  $\Bk_{W \cap W'} \geq 3$.  Hence by Lemma \ref{critcapcup}, $\Bk_{W \cup W'} + \Bk_{W \cap W'} \leq \Bk_{W} + \Bk_{W'} = 8$, and so $\Bk_{W \cap W'},\Bk_{W \cup W'} = 4$.
As a result of this, we may always take 4-edge-critical sets of $C_2$-tight graphs to be symmetric (in particular we may assume $W$ is symmetric).

There are either 0, 2, or 4 loops on the vertices of $W$ (by symmetry and $G$ being $C_2$-tight), with $W$ being 4, 2, or 0-critical respectively in those cases.
Again, since $k_{X \cup W} = k_X + k_W - k_{X \cap W} - d(X,W)$, any looped vertex of $W$ is also a vertex of $X \cap W$, giving $k_W - k_{X \cap W} = 0$ ($X \cap W$ contains at least 2 vertices and is symmetric, so has at least 4 for its edge-critical value) and $X \cup W \cup \{v\}$ is not sparse. 
Hence the lemma holds in this case with $x_1, x_2 = y,z$.

Now assume neither $x,y$ nor $y,z$ are contained in $0$-critical sets.
If both $x,y,x',y'$ and $y,z,y',z'$ are in $4$-edge-critical sets $W_1$, $W_2$, then we may suppose $W_1$ and $W_2$ are symmetric as before.
Both $W_1 \cup W_2$ and $W_1 \cap W_2$ are  $C_2$-symmetric and have at least $2$ vertices, so are also both $4$-edge-critical.
In particular, $W_1 \cup W_2 \cup \{v,v'\}$ violates the sparsity of $G$.
When only one of $x,y,x',y'$ and $y,z,y',z'$ are in a $4$-edge-critical set, say $x,y,x',y'$ in $W$, assume there exists a $3$-edge-critical set $U$ containing $y,z$.
Lemma \ref{critcapcup} for edge-criticality says  
$$\Bk_{U \cup W} = \Bk_U + \Bk_W - \Bk_{U \cap W} - d(U,W).$$ 
Since $\Bk_{U \cap W} \geq 2$ and $d(U,W) \geq 1$ (due to the presence of the edge $xz$), we have $\Bk_{U \cup W} \leq 4$.
Similarly, $\Bk_{U \cup U' \cup W} \leq 4$, but adding $v, v'$ contradicts the sparsity of $G$.
Finally, if no such $4$-edge-critical set exists, $x,y$ and $y,z$ cannot both  be in 3-edge-critical sets, say $U_1, U_2$, by considering $U_1 \cup U_2 \cup \{v\}$, as $xz$ again gives $d(U_1,U_2) \geq 1$.
Hence the lemma holds when one edge is present on the neighbours of $v$.

\emph{Case 3:} Lastly suppose there are no edges induced by the vertices of $N(v)$.
It is easy to show that there exists a $0$-critical set on at most one pair of neighbours of $v$, while there can be 3-edge-critical sets on at most two pairs.

First assume $x,y$ is contained in a $0$-critical set, $X$.
Assume $x,z$ and $y,z$ are contained in $3$-edge-critical sets $U_1, U_2$ and note that $(U_1 \cup U_2) \cap X$ and any supersets thereof (since they contain $x,y$) cannot be 3-edge-critical else all 3 pairs of neighbours of $v$ are in 3-edge critical sets.
Consider the equations,
\begin{equation}\label{eq: bar equality}
    \Bk_{U_1 \cup U_2} = \Bk_{U_1} + \Bk_{U_2} - \Bk_{U_1 \cap U_2} - d(U_1,U_2)
\end{equation}
\begin{equation}\label{eq: no bar equality}
    k_{(U_1 \cup U_2)\cup X} = k_{U_1 \cup U_2} + k_{X} - k_{(U_1 \cup U_2) \cap X} - d(U_1 \cup U_2, X).
\end{equation}
A contradiction occurs when  
$$k_{U_1 \cup U_2} \leq k_{(U_1 \cup U_2) \cap X} + d(U_1 \cup U_2, X),$$ 
as this would violate the sparsity of $G$.
Let $l_a \geq l_b$ count the number of loops on vertices of $U_1 \cup U_2, (U_1 \cup U_2)\cap X$ respectively.
Equation (\ref{eq: bar equality}) along with the inequalities $\Bk_{U_1 \cup U_2} \geq 4$ and $\Bk_{U_1 \cap U_2} \geq 2$ gives that these two inequalities are in fact equalities.
Then, since $\Bk_{(U_1 \cup U_2) \cap X} \geq 4 = \Bk_{(U_1 \cup U_2)}$, we have $$k_{(U_1 \cup U_2) \cap X} = \Bk_{(U_1 \cup U_2) \cap X} - l_b \geq \Bk_{(U_1 \cup U_2)} - l_a = k_{(U_1 \cup U_2)},$$ 
which, with Equation (\ref{eq: no bar equality}), gives the desired contradiction.
So there is a pair $x_1, x_2$ not contained in a 0-critical or 3-edge critical set.
By Remark \ref{rem: no 0c and 4e}, there is no 4-edge-critical or 1-critical set containing $x_1,x_2,x_1',x_2'$, completing this case.

Finally, suppose there is no 0-critical set on any pair of neighbours of $N(v)$.
We know, from above, that at most one pair of neighbours of $v$ with their symmetric copies can be contained in a 4-edge-critical set.
Hence in this case we only obtain a contradiction to the lemma if say $x, y, x', y'$ is in a 4-edge-critical set $W$ while $x, z$ and $y, z$ are in 3-edge-critical sets, say $U_1$ and $U_2$ respectively (see Figure \ref{fig: U_1 U_2 W}).
Recall that $x,y$ cannot also be in a 3-edge-critical set and since $\Bk_{U_1 \cap U_2} \geq 2$ in Equation (\ref{eq: bar equality}), we have $\Bk_{U_1 \cup U_2} = 4$.
Further, $x,y \in (U_1 \cup U_2)\cap W$, so $\Bk_{(U_1 \cup U_2) \cup W} = 4$.
Now instead of considering the 3-edge-critical sets $U_1$ and $U_2$ with the 4-edge-critical set $W$, consider the 3-edge-critical  sets $U_1'$ and $U_2'$ with the 4-edge-critical set $(U_1 \cup U_2) \cup W$. The same methods show $\Bk_{(U_1 \cup U_2) \cup W \cup (U_1' \cup U_2')}= 4$.
However the set $(U_1 \cup U_2) \cup W \cup (U_1' \cup U_2')$ with $\{v,v'\}$ added is not sparse.
This exhausts all cases, completing the proof.
\end{proof}

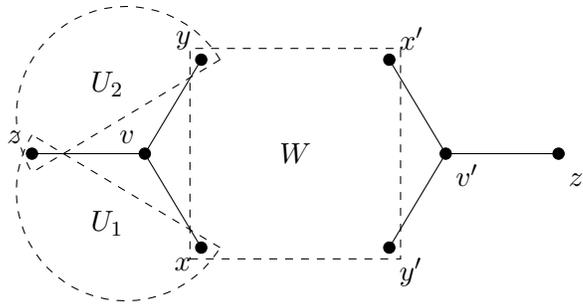
\begin{figure}[ht]
    \centering
    \begin{tikzpicture}
  [scale=.5,auto=left]
  
  \coordinate (n1) at (0,3);
  \coordinate (n2) at (3,3);
  \coordinate (n3) at (4.5,0.5);
  \coordinate (n4) at (4.5,5.5);
  \coordinate (n5) at (9.5,0.5);
  \coordinate (n6) at (9.5,5.5);
  \coordinate (n7) at (11,3);
  \coordinate (n8) at (14,3);
  \coordinate (n9) at (7,3);
  \coordinate (n10) at (2,1.15);
  \coordinate (n11) at (2,4.85);

 \draw[dashed] (4.2,0.2) rectangle (9.8,5.8);
 \draw[fill=black] (n1) circle (0.15cm)
	    node[above left] {$z$};
 \draw[fill=black] (n2) circle (0.15cm)
	    node[above left] {$v$};
 \draw[fill=black] (n3) circle (0.15cm)
	    node[below left] {$x$};
 \draw[fill=black] (n4) circle (0.15cm)
	    node[above left] {$y$};
 \draw[fill=black] (n5) circle (0.15cm)
	    node[below right] {$y'$};
 \draw[fill=black] (n6) circle (0.15cm)
	    node[above right] {$x'$};
 \draw[fill=black] (n7) circle (0.15cm)
	    node[below right] {$v'$};
 \draw[fill=black] (n8) circle (0.15cm)
	    node[below right] {$z'$};
 \draw (n9) node {$W$};
 \draw (n10) node {$U_1$};
 \draw (n11) node {$U_2$};
 \draw[dashed] (5,5.5) arc (30.964:210.964:2.9155);
 \draw[dashed] (5,5.5) -- (0,2.5);
 \draw[dashed] (0,3.5) arc (149.036:329.036:2.9155);
 \draw[dashed] (5,0.5) -- (0,3.5);

  \foreach \from/\to in {n1/n2,n2/n3,n2/n4,n5/n7,n6/n7,n7/n8} 
    \draw (\from) -- (\to);

\end{tikzpicture}
    \caption{Diagram of part of a $C_2$-tight graph with hypothetical 3-edge-critical sets $U_1, U_2$ and 4-edge-critical set $W$, with $x,z \in U_1$, $y,z \in U_2$, $x,y,x',y' \in W$.}
    \label{fig: U_1 U_2 W}
\end{figure}

\begin{lem}\label{lem: c2 x=z'}
Let $(G,\phi)$ be $C_2$-tight and suppose $v \in V$ is a node such that $N(v) = \{x,y,x'\}$ and $N(v) \cap N(v') = \{x,x'\}$.
For a pair $(x_1,x_2) \in \{(x,y),(x',y)\}$, 
with $x_1 x_2\notin E$, we have:\\
$x_1, x_2$ is not contained in any $0$-critical set or any $3$-edge-critical set; and \\
$x_1, x_2, x_1', x_2'$ is not contained in any $1$-critical set or any $4$-edge-critical set.
\end{lem}

\begin{proof}
Without loss of generality, one of the following hold:
\begin{enumerate}
    \item $xy, x'y' \in E$, $x'y, xy' \notin E$.
    \item $xy, x'y, xy', x'y' \notin E$.
\end{enumerate}
The edge sets described in (1) and (2) describe all possibilities when $N(v) \cap N(v') = \{x,x'\}$.
Suppose the edges present are as in (1) or (2), and at least one of the following exists: (i) there exists $X$ with $x',y \in X$ which is $0$-critical; or (ii) there exists $W$ with $x',y, x, y' \in W$ which is 1-critial or 4-edge-critical.
Noting that $X\cup X'$ is $0$-critical and contains all the neighbours of $v$ and $v'$, as with such a $W$, $X \cup \{v,v'\}$ and $W \cup \{v,v'\}$ break sparsity of $G$ immediately.

Hence for the remainder of the proof, we need only consider 3-edge-critical sets.
First in case (1), let $U$ be 3-edge-critical with $x',y \in U$.
From Lemma \ref{critcapcup}, 
$$\Bk_{U \cup U'} \leq \Bk_U + \Bk_{U'} - d(U,U') = 3 + 3 -2 =4.$$ 
Then $U \cup U' \cup \{v,v'\}$ violates the sparsity of $G$, completing the first case.
Now assume we have no edges as in (2).
We want to show that we can add either $xy,x'y'$ or $x'y,xy'$ to $G-\{v,v'\}$.
Suppose there exists two 3-edge-critical sets, $U_1, U_2$, with $x,y \in U_1$ and $x',y \in U_2$.
Relabel $Y:= U_1 \cup U_2$.
Note that $Y$ is not 3-edge-critical else $Y + \{v\}$ is not sparse.
On the other hand $\Bk_{U_1 \cap U_2} \geq 2$, so by counting $\Bk_{Y} \leq 4$, hence $Y$ must be 4-edge-critical.
Symmetrically, $\Bk_{Y'} = 4$.
Then as $x,x' \in Y \cap Y'$, $\Bk_{Y \cap Y'} \geq 3$, and since $Y \cap Y'$ is symmetric, the number of edges on the set must be even, so $\Bk_{Y \cap Y'} \geq 4$. Hence 
$$\Bk_{Y \cup Y'} = \Bk_Y + \Bk_{Y'} - \Bk_{Y \cap Y'} - d(Y,Y') \leq 4+4-4=4, $$
so adding $\{v,v'\}$ to $Y\cup Y'$ breaks sparsity of $G$.
\end{proof}

\begin{figure}[ht]
    \centering
    \begin{tikzpicture}
  [scale=.4,auto=left]
	    
  \coordinate (n21) at (10,8);
  \coordinate (n22) at (12.5,2.5);
  \coordinate (n23) at (12.5,5.5);
  \coordinate (n24) at (15.5,5.5);
  \coordinate (n25) at (18,0);
  \coordinate (n26) at (12.5,2.5);
  \coordinate (n27) at (15.5,2.5);
  \coordinate (n28) at (15.5,5.5);
  \coordinate (n29) at (14,4);
  
 \draw[dashed, fill=lightgray] (n29) circle (4cm); 
 \draw[fill=black] (n21) circle (0.15cm)
	    node[left] {$v$};
 \draw[fill=black] (n22) circle (0.15cm)
	    node[below left] {$x$};
 \draw[fill=black] (n23) circle (0.15cm)
	    node[below right] {$y$};
 \draw[fill=black] (n24) circle (0.15cm)
	    node[above right] {$x'$};
 \draw[fill=black] (n25) circle (0.15cm)
	    node[right] {$v'$};
 \draw[fill=black] (n26) circle (0.15cm);
 \draw[fill=black] (n27) circle (0.15cm)
	    node[above left] {$y'$};
 \draw[fill=black] (n28) circle (0.15cm);

  \foreach \from/\to in { n21/n22,n21/n23,n21/n24,n25/n26,n25/n27,n25/n28,n23/n24,n26/n27} 
    \draw (\from) -- (\to);

  \coordinate (n31) at (20,8);
  \coordinate (n32) at (22.5,2.5);
  \coordinate (n33) at (22.5,5.5);
  \coordinate (n34) at (25.5,5.5);
  \coordinate (n35) at (28,0);
  \coordinate (n36) at (22.5,2.5);
  \coordinate (n37) at (25.5,2.5);
  \coordinate (n38) at (25.5,5.5);
  \coordinate (n39) at (24,4);
  
 \draw[dashed, fill=lightgray] (n39) circle (4cm); 
 \draw[fill=black] (n31) circle (0.15cm)
	    node[left] {$v$};
 \draw[fill=black] (n32) circle (0.15cm)
	    node[below left] {$x$};
 \draw[fill=black] (n33) circle (0.15cm)
	    node[below right] {$y$};
 \draw[fill=black] (n34) circle (0.15cm)
	    node[above right] {$x'$};
 \draw[fill=black] (n35) circle (0.15cm)
	    node[right] {$v'$};
 \draw[fill=black] (n36) circle (0.15cm);
 \draw[fill=black] (n37) circle (0.15cm)
	    node[above left] {$y'$};
 \draw[fill=black] (n38) circle (0.15cm);

  \foreach \from/\to in { n31/n32,n31/n33,n31/n34,n35/n36,n35/n37,n35/n38} 
    \draw (\from) -- (\to);

\end{tikzpicture}
    \caption{The local structure of cases (1) and (2) in Lemma \ref{lem: c2 x=z'}.}
    \label{fig CI1.5}
\end{figure}
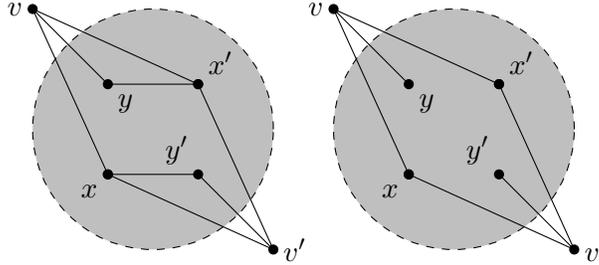

\begin{lem}\label{lem: c_2 three edges redu}
Let $(G,\phi)$ be a $C_{2}$-tight graph containing a node $v$ with three distinct neighbours.
Then $G \setminus \{v,v'\} + \{x_1 x_2, x_1'x_2'\}$ is $C_{2}$-tight for some $x_1, x_2 \in N(v)$.
\end{lem}

\begin{proof}
From the definition of a $C_2$-tight graph, we know there are no fixed vertices or edges in $G$.
Therefore $vv' \notin E$ and $|N(v) \cap N(v')|$ cannot be 1 or 3, as this would require a fixed vertex.
Lemmas \ref{lem: c2 0,1,3,4 crit} and \ref{lem: c2 x=z'} that for the remaining possible cases, that is when $N(v) \cap N(v') = \emptyset, \{x,x'\}$ the reduction will preserve sparsity, as required.
\end{proof}

\begin{lem}\label{lem: c_2 2 edges 1 loop redu}
Let $(G,\phi)$ be a $C_{2}$-tight graph.
Suppose $v$ is a node adjacent to distinct vertices $\{x,y\}$, and $v$ is incident to a loop.
There exists some $x_1 \in \{x,y\}$ such that $G \setminus \{v,v'\} + \{(x_1,x_1),(x_1',x_1')\}$ is $C_2$-tight.
\end{lem}

\begin{proof}
    Suppose there exists 0-critical sets $X,Y$, with $x\in X$ and $Y \in Y$.
    Then $X \cup Y +\{v\}$ is not sparse.
    Hence, without loss of generality we may take it that $x$ is not in a 0-critical set.
    Instead now suppose $x,x'$ is in a 1-critical set, say $W$.
    Then $W \cup W$ and $W \cap W$ are both contain $x,x'$ and are $c_2$-symmetric, hence have an even critical value.
    From the equation in Lemma \ref{critcapcup}, 
    $$2= k_W + k_{W'} = k_{W\cup W'} + k_{W\cap W'} + d(W,W').$$
    Thus one of $W \cup W$ and $W \cap W$ is 0-critical, label this 0-critical set $U$.
    If $y$ is in a 0-critical set $Y$, $U\cup Y \cup \{v\}$ breaks sparsity of $G$.
    If instead $y,y'$ is in a 1-critical set, we may do the same steps as above to find a 0-critical set containing both $y,y'$, leading to the same contradiction.
    Hence there is a $x_1 \in \{x,y\}$ which is not in a 0-critical set with $x_1, x_1'$ not in a 1-critical set, completing the proof.
\end{proof}

We put together the combinatorial analysis to this point to prove the following recursive construction. From this we then deduce our characterisation of $C_2$-isostatic graphs. We need one more lemma which we prove for arbitrary cyclic groups as we will use it again later in the article.

Two vertices $u,v \in V$ are connected if there exists a path from $u$ to $v$.
We define two vertices $u,v \in V$ to be \emph{$\gamma$-symmetrically connected} if $u = \gamma v$ or $v = \gamma u$, or if there exists a path from $u$ to $\gamma v$ or $\gamma u$ to $v$.
A \emph{$\Gamma$-symmetrically connected component} is a set of vertices which are pairwise $\gamma$-symmetrically connected for some $\gamma \in \Gamma$, and a graph is \emph{$\Gamma$-symmetrically connected} if it has only one $\Gamma$-symmetrically connected component.

\begin{lem}\label{lem: components}
    A graph $(G,\phi)$ is $C_n$-tight if and only if every $C_n$-symmetrically connected component of $G$ is $C_n$-tight.
\end{lem}

\begin{proof}
    Label the $C_n$-symmetrically connected components of $G$ with $H_1, \dots, H_r$, with $H_i = (V_i, E_i, L_i)$.
    By sparsity, we know each of the subgraphs $H_i$ are sparse.
    We have $|E_i|+|L_i| \leq 2|V_i|$ and $|E_i| \leq 2|V_i|-3$, which gives $G$ being tight if and only if equality hold in the first equation for each $i$, which is to say each $C_n$-symmetrically connected component is tight.
\end{proof}

We recall the following $C_2$-tight graphs: $(\mathcal{P}_1,\phi_0)$ with one fixed vertex and two fixed loops; $(\mathcal{P}_1,\phi_0)$ with two vertices and four loops and no vertices or loops fixed.
These graphs are depicted in Figure \ref{fig: c_2 base}.

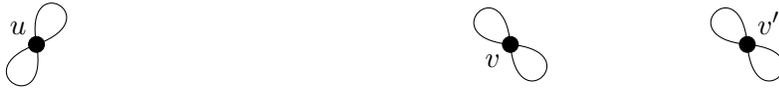
\begin{figure}[ht]
    \centering
    \begin{tikzpicture}
  [scale=0.7,auto=left]

  \coordinate (n0) at (-9,0);
  \coordinate (n1) at (0,0);
  \coordinate (n2) at (4.5,0);  

 \draw[fill=black] (n0) circle (0.15cm)
    node[above left] {$u$};
 \draw[fill=black] (n1) circle (0.15cm)  
    node[below left] {$v$};
 \draw[fill=black] (n2) circle (0.15cm)
    node[above right] {$v'$};

\draw[scale = 3] (-3,0)  to[in=20,out=100,loop] (0,0);
\draw[scale = 3] (-3,0)  to[in=200,out=280,loop] (0,0);
\draw[scale = 3] (0,0)  to[in=95,out=175,loop] (0,0);
\draw[scale = 3] (0,0)  to[in=-5,out=-85,loop] (0,0);
\draw[scale = 3] (1.5,0)  to[in=95,out=175,loop] (0,0);
\draw[scale = 3] (1.5,0)  to[in=-5,out=-85,loop] (0,0);

\end{tikzpicture}
    \caption{The base graphs, left $(\mathcal{P}_1,\phi_0)$ and right $(\mathcal{P}_2,\phi_2)$, for $C_2$-symmetric linearly constrained frameworks.}
    \label{fig: c_2 base}
\end{figure}

\begin{thm}\label{thm:c2recursion}
A graph $(G,\phi)$ is $C_2$-tight if and only if $(G,\phi)$ can be generated from disjoint copies of $(\mathcal{P}_1,\phi_0)$ and $(\mathcal{P}_2,\phi_2)$ by symmetrised 0-extensions and symmetrised 1-extensions.
\end{thm}

\begin{proof}
It is easy to see that any graph generated from $(\mathcal{P}_1,\phi_0)$ or $(\mathcal{P}_2,\phi_2)$ by symmetrised 0-extensions and 1-extensions is $C_2$-symmetrically connected and $C_2$-tight. Hence $G$ is $C_2$-tight by Lemma \ref{lem: components}.

We show by induction that any $C_2$-tight graph $G$ can be generated from one of $(\mathcal{P}_1,\phi_0),(\mathcal{P}_1,\phi_0)$.
We may assume by Lemma \ref{lem: components} that $G$ is $C_2$-symmetrically connected.
Suppose the induction hypothesis holds for all graphs with $|V| < m$.
Now let $|V|= m$ and suppose $G$ is not isomorphic to $(\mathcal{P}_1,\phi_0)$ or $(\mathcal{P}_2,\phi_2)$.
Since $G$ is $C_2$-tight, the minimum degree is at least 2 and at most 4.

A degree 2 vertex in a tight graph must be adjacent to two vertices and hence is reducible by Lemma \ref{lem: c_2 two edges redu}.
A degree 3 vertex can have a loop and an edge incident to it or be adjacent to three vertices. The former is reducible by Lemma \ref{lem: c_2 two edges redu} and the latter by Lemma \ref{lem: c_2 three edges redu}.

Hence suppose $\delta(G)=4$ and $v$ is a vertex of minimum degree.
We claim there is such a $v$ that is incident to a loop. Suppose not, then every vertex has at least 4 neighbours and so $2|E| \leq \sum_{v \in V}\deg_{E} (v)$ which violates the definition of tight. Since $v$ is incident to a loop and has degree 4 either it has two incident loops or is incident to two edges and a loop. In the former case the orbit of such a vertex would be it's own $C_2$-symmetrically connected component.
In the latter case $w$ is reducible by Lemma \ref{lem: c_2 2 edges 1 loop redu}.
This exhausts the possible cases and completes the proof.
\end{proof}

\begin{thm}\label{thm:mainc2}
A graph $(G,\phi)$ is $C_2$-isostatic if and only if it is $C_2$-tight.
\end{thm}

\begin{proof}
Necessity was proved in Theorem \ref{thm: fixed 2dim lin}.
Lemma \ref{lem: base rigid} implies the base graphs $(\mathcal{P}_1,\phi_0)$ and $(\mathcal{P}_2,\phi_2)$ are $C_2$-isostatic.
Sufficiency follows from Theorem \ref{thm:c2recursion} and  Lemmas \ref{lem: c_2 two edges redu}, \ref{lem: c_2 three edges redu} and \ref{lem: c_2 2 edges 1 loop redu} by induction on $|V|$.
\end{proof}

\section{$C_n$-symmetric isostatic graphs}
\label{sec:Cn}

In this section we analyse $C_n$ rotational symmetry, for all odd $n \geq 3$. The arguments will build on the proofs in the previous section.

\subsection{Reduction Operations}

\begin{lem}\label{lem: c_n two edges redu}
Let $(G,\phi)$ be a $C_{n}$-tight graph containing a vertex $v$ incident to two edges.
Then $G \setminus \bigcup_{i=0}^n\{\gamma^i v\}$ is $C_{n}$-tight.
\end{lem}

Note that this lemma includes the cases when $v$ is degree two adjacent to two distinct vertices and when $v$ is degree three with a loop at $v$, and is depicted in Figure \ref{fig:0ext}.

\begin{proof}
For fixed $1 \leq k \leq n$ and any $0 \leq i_1 < \dots < i_k \leq n-1$, if $G \setminus \bigcup_{j=0}^k\{\gamma^{i_j} v\}$ breaks sparsity, $G$ would not be sparse.
\end{proof}

Let us first show the following key technical lemma for $C_3$, and then follow it with the general case.
Both proofs are similar, but technical; the easier case is presented for the reader's convenience.

\begin{lem}
Let $(G,\phi)$ be a $C_{3}$-tight graph containing a $v$ adjacent to two distinct vertices $\{v_1,v_2\}$ and incident to a loop.
Let $\gamma \in C_3 \setminus \{id\}$.
Then there exists some $x \in \{v_1,v_2\}$ such that $G \setminus \bigcup_{i=0}^2\{\gamma^i v\} + \bigcup_{i=0}^2\{\gamma^i x \gamma^i x\}$ is tight.
\end{lem}

\begin{proof}
We prove the lemma by first checking there is a neighbour of $v$ not contained in a $0$-critical subset of $V$, and then (for any $k \in \{1,2\}$) that any $k+1$ symmetric copies of that neighbour are not contained in a $k$-critical subset of $V$.

If both of $v_1,v_2$ were in 0-critical sets, say $V_1, V_2$ respectively, then $V_1 \cup V_2 \cup \{v\}$ would break sparsity.
Label the vertex which is not in a 0-critical set $x$.
Note that any symmetric copy of $x$ is also not in a 0-critical set.

Let $W$ be a set which contains $\{\gamma^{i_1}x, \gamma^{i_2}x\}$.
First we note that $W$ and $\gamma^{i_2 - i_1}W$ are not $0$-critical since they each contain a copy of $x$.
Suppose for a contradiction that $W$ and $\gamma^{i_2 - i_1}W$ are $1$-critical.
Then $W \cup \gamma^{i_2 -i_1} W$ contains $\{\gamma^{i_1}x, \gamma^{i_2}x, \gamma^{2i_2 - i_1}x\}$ and is not $0$-critical;
and $W \cap \gamma^{i_2 -i_1} W$ contains $\{\gamma^{i_2}x\}$ and is also not $0$-critical.
By Lemma \ref{critcapcup}, $k_{W \cup \gamma^{i_2 -i_1} W}$ and $k_{W \cap \gamma^{i_2 -i_1} W}$ must both be $1$.
Observe that $\gamma^{2i_2 -2i_1}$ is the third and final element of $C_3$, and $\gamma^{2i_2 -2i_1} W$ is $1$-critical, so again, by Lemma \ref{critcapcup}, $k_{(W \cup \gamma^{i_2 -i_1} W}) \cup \gamma^{2i_2 -2i_1} W$ and $k_{(W \cup \gamma^{i_2 -i_1} W) \cap \gamma^{2i_2 -2i_1} W}$ must both be $1$.
However, since $G$ is $C_3$-symmetric, 3 divides both $|W \cup \gamma W \cup \gamma^2 W|$ and $i_{E+L}(W \cup \gamma W \cup \gamma^2 W)$.
Then since $i_{E+L}(A) = 2|A| - k_A$, $3$ would have to divide the criticality of $W \cup \gamma W \cup \gamma^2 W$, which is a contradiction.
Therefore any $\{\gamma^{i_1}x, \gamma^{i_2}x\}$ with distinct $i_1, i_2 \in \{0,1,2\}$ is not contained in a $1$-critical set.

Let $W$ be a set which contains $\{x, \gamma x, \gamma^{2}x\}$, we know from the above, $W$ is not $0 $ or $1$-critical. Assume for contradiction it is $2$-critical.
Then $W \cup \gamma W$ contains $\{x, \gamma x, \gamma^{2}x\}$ and so is not $1$-critical;
and $W \cap \gamma W$ contains $\{x, \gamma x, \gamma^{2}x\}$ which is also not $1$-critical.
By Lemma \ref{critcapcup}, $k_{W \cup \gamma W}$ and $k_{W \cap \gamma W}$ must both be $2$.
Similarly, $W \cup \gamma W \cup \gamma^2 W$ is $2$-critical.
However, as before $G$ is $C_3$-symmetric, so $3$ would have to divide the criticality of $W \cup \gamma W \cup \gamma^2 W$, which is a contradiction.
This proves any set containing $\{x, \gamma x, \gamma^{2}x\}$ is not a $2$-critical set.
Hence, there are no sets which would break sparsity conditions of our graph class by performing a $1$-reduction at $v$, adding a loop at a neighbour.
\end{proof}

Having established this result for $C_3$-symmetric graphs, we extend this to odd order cyclic symmetric groups.

The two 1-reductions we perform have similarities in their proofs, where we build an inductive argument on the number of symmetric copies of neighbours of $v$.
Since some of these ideas overlap, they will be shared between the two proofs where possible.
The significant difference between the two proofs is that adding an edge means we must check both edge sparsity and sparsity are preserved, whereas adding a loop does not require the edge sparsity condition to be checked.
To begin with we derive a technical lemma which is used in the inductive step of both arguments.

Define $X_j(A) = \big\{\gamma^{i_1}A, \dots, \gamma^{i_j}A : i_1 < \dots < i_j \in \{0, \dots, n-1\}\big\}$, where $A\subset V$.
Write $\mathcal{X}_{j}(A)$ for the set of all such subsets of size $j$ for the given vertex set $A$.
We will write $X_j$ for $X_j(A)$ and $\mathcal{X}_j$ for $\mathcal{X}_j(A)$ where the context is clear.
Let $\phi: \mathcal{X}_{j} \to \Z_{n}$, $\phi(\{\gamma^{i_1}A, \dots, \gamma^{i_j}A\}) = \{i_1, \dots, i_j\}$.

\begin{lem}\label{lem: X_k induction}
Let $(G,\phi)$ be a $C_{n}$-tight graph
and let $X_j$ be defined as above.
If for $A\subset N(v)$, $X_1$ is not in a $0$-critical set for all $X_1 \in \mathcal{X}_{1}$, then $X_j$ is not contained in a $(j-1)$-critical set for all $j\leq n$.
\end{lem}

\begin{proof}
We proceed by induction on $j$ and begin by noting that the case when $j=1$ is trivial.
Assume no $X_j$ is contained in a $(j-1)$-critical set for all $j\leq k$.
We will show any $X_{k+1}$ is not contained in a $k$-critical set.
Fix $X_{k+1} \in \mathcal{X}_{k+1}$ and write $X = \phi(X_{k+1})$.
For notation, write $\gamma^i X = \phi(\gamma^i X_{k+1})$ and $\gamma^i X \cup \gamma^j X = \phi(\gamma^i X_{k+1} \cup \gamma^j X_{k+1})$ and $\gamma^i X \cap \gamma^j X$ similarly.
By the induction hypothesis, any set containing $X_{k+1}$ also contains $k$ or fewer copies of $x$, so is not $j$-critical for any $j<k$.
Suppose for a contradiction that $X_{k+1}$ is contained in a $k$-critical set, say $W$.
We will show that $\bigcup_{i=0}^{n-1}\gamma^i W$ has $2|\bigcup_{i=0}^{n-1}\gamma^i W|-a$ edges for some $a<n$. 
Since $\bigcup_{i=0}^{n-1}\gamma^i W$ is symmetric, $n$ divides $|\bigcup_{i=0}^{n-1}\gamma^i W|$. As our graph class has no fixed edges or loops, this gives a contradiction.

Observe that if 
$\gamma^t X \subseteq X \cup \dots \cup \gamma^{t-1} X$ and
 $X \cup \dots \cup \gamma^{t-1} X = X \cup \dots \cup \gamma^t X$ then 
 $$k_{W \cup \dots \cup \gamma^{t-1} W} = k_{W \cup \dots \cup \gamma^t W}.$$
Therefore, to bound $k_{W \cup \dots \cup \gamma^{n-1} W}$ our main focus is the case when $\gamma^t X \not\subseteq X \cup \dots \cup \gamma^{t-1} X$.
First, when $\gamma X \not\subseteq X$ we have $|X \cap \gamma X| < |X| < |X \cup \gamma X|$.
By the induction hypothesis, any set containing $\phi^{-1}(X \cap \gamma X)$, such as $W\cap \gamma W$, is not contained in a $(|X\cap \gamma X|-1)$-critical set.
This implies $k_{W\cap \gamma W} \geq |X \cap \gamma X|$.
Lemma \ref{critcapcup} implies that
\begin{equation}\label{eq: W and X}
    \begin{split}
        k_{W \cup \gamma W} & = k_W + k_{\gamma W} - k_{W \cap \gamma W} - d(W, \gamma W) \\
        & \leq |X| -1 + |\gamma X| -1 - |X \cap \gamma X| - d(W, \gamma W) \\
        & \leq |X \cup \gamma X| -2
    \end{split}
\end{equation}
and hence the critical value for $W \cup \gamma W$ is at most $(|X \cup \gamma X| -2)$.\\
We repeat this process with $\gamma^2 X \not\subseteq X \cup \gamma X$, and hence investigate $(X \cup \gamma X) \cup \gamma^2 X$ and $(X \cup \gamma X) \cap \gamma^2 X$.
We know that $k_{W \cup \gamma W} = k$ or $k_{W \cup \gamma W} \leq |X \cup \gamma X| -2$ and $k_{\gamma^2 W}=k$.
Since $\gamma^2 X \not\subseteq X \cup \gamma X$ we have 
$$|(X \cup \gamma X) \cap \gamma^{2} X| < |X \cup \gamma X| < |(X \cup \gamma X) \cup \gamma^{2} X|.$$
Noting that $\phi^{-1}((X \cup \gamma X) \cap \gamma^{2} X)$ is not contained in any $(|(X \cup \gamma X) \cap \gamma^{2} X| -1)$-critical set, (as $(X \cup \gamma X) \cap \gamma^{2} X \subset \gamma^{2}X$ we can apply the induction hypothesis), so 
$$k_{(W \cup \gamma W) \cap \gamma^2 W} \geq |(X \cup \gamma X) \cap \gamma^{2} X|$$
(as $\phi^{-1}((X \cup \gamma X) \cap \gamma^{2} X) \subseteq (W \cup \gamma W) \cap \gamma^2 W$), therefore
\begin{equation}\label{eq: 2nd W and X}
    \begin{split}
        k_{W \cup \gamma W\cup \gamma^{2}W} & = k_{W\cup \gamma W} + k_{\gamma^{2} W} - k_{(W \cup \gamma W) \cap \gamma^2 W} - d(W \cup \gamma W, \gamma^2 W) \\
        & \leq |X \cup \gamma X| -2 + |\gamma^2 X| -1 - |(X \cup \gamma X) \cap \gamma^2 X| - d(W \cup \gamma W, \gamma^2 W) \\
        & \leq |X \cup \gamma X \cup \gamma^2 X| -3.
    \end{split}
\end{equation}
Thus the critical value for $W \cup \gamma W \cup \gamma^2 W$ is at most $(|X \cup \gamma X \cup \gamma^2 X| -3)$.
Recalling that $\gamma^t X \subseteq X \cup \dots \cup \gamma^{t-1} X$ implies that 
$$k_{W \cup \dots \cup \gamma^{t-1} W} = k_{W \cup \dots \cup \gamma^t W},$$ 
and noting that $|(X \cup \dots \cup \gamma^{n-1}X)\setminus X| = n-k-1$, the case when $\gamma^t X \not\subseteq X \cup \dots \cup \gamma^{t-1} X$ can happen at most $n-k-1$ times.
Therefore, we obtain that 
$$k_{W \cup \dots \cup \gamma^{n-1} W} \leq |X \cup \dots \cup \gamma^{n-1} X| - n+k.$$
Finally since $|X \cup \dots \cup \gamma^{n-1} X|=n$, $W \cup \dots \cup \gamma^{n-1} W$ cannot be $n$-critical.
However, since $G$ is $C_n$-symmetric, $n$ divides $|W \cup \dots \cup \gamma^{n-1} W|$ and $n$ divides $i_{E+L}(W \cup \dots \cup \gamma^{n-1} W)$. However then, since $i_{E+L}(A) = 2|A| - k_A$, $n$ divides the criticality of $W \cup \dots \cup \gamma^{n-1} W$, a contradiction.
\end{proof}

\begin{lem}\label{lem: c_n 2 edges 1 loop redu}
Let $(G,\phi)$ be a $C_{n}$-tight graph.
Suppose $v$ is a node adjacent to distinct vertices $\{v_1,v_2\}$, and $v$ has a loop. Let $\gamma \in C_n$ be a generator of $C_n$.
There exists some $x \in \{v_1,v_2\}$ such that $G \setminus \bigcup_{i=0}^{n-1}\{\gamma^i v\} + \bigcup_{i=0}^{n-1}\{\gamma^i x \gamma^i x\}$ is $C_n$-tight.
\end{lem}

\begin{proof} Since $\gamma \in C_n$ is a generator of $C_n$, for any $u\in V$, $u, \gamma u, \ldots, \gamma^{n-1} u$ are all distinct.
We will show that there is a neighbour of $v$ not contained in a $0$-critical subset of $V$, and that (for any $k \in \{1,\dots, n-1\}$) any $k+1$ symmetric copies of that neighbour are not contained in a $k$-critical subset of $V$.
We prove this by induction on $k$.

To see there exists an $x \in \{v_1,v_2\}$ which is not in a 0-critical set, suppose both $v_1$ and $v_2$ are in 0-critical sets. Let $V_1, V_2$ denote the 0-critical sets containing $v_1,v_2$ respectively. Then $V_1 \cup V_2 \cup \{v\}$ would break sparsity.
The base case of induction is now complete since, by extension, $\gamma^i x$ is not in a $0$-critical set for any $i \in [n-1]$.
Indeed, we now have a set $\{x, \gamma x, \dots, \gamma^{n-1} x\}$ where no one element is contained in a 0-critical set.
Hence, we have, for $\{x\}\subset N(v)$, shown that any $X_1(\{x\}) \in \mathcal{X}_1(\{x\})$ is not contained in a 0-critical set, so Lemma \ref{lem: X_k induction} implies that no $X_j$ is contained in a $(j-1)$-critical set for all $j\leq n$.
Hence, the 1-reduction will not break sparsity of $G$, completing the proof.
\end{proof}

\begin{lem}\label{lem: c_n 3edges redu}
For a positive odd integer $n$, let $(G,\phi)$ be a $C_{n}$-tight graph 
and let $v\in V$ be a node adjacent to distinct vertices $\{v_1,v_2,v_3\}$.
Let $\gamma \in C_n$ be a generator of $C_n$.
There exists some $x,y \in \{v_1,v_2,v_3\}$ such that $G \setminus \bigcup_{i=0}^{n-1}\{\gamma^i v\} + \bigcup_{i=0}^{n-1}\{\gamma^i x \gamma^i y\}$ is $C_n$-tight.
\end{lem}

\begin{proof}
With the same arguments used in Lemma \ref{lem: c2 0,1,3,4 crit}, we may show that there exists a pair $\{x_0, y_0\} \in \{\{v_1, v_2\}, \{v_1, v_3\}, \{v_2, v_3\}\}$ which is not contained in a 0-critical set or a 3-edge-critical set. 
We claim that for some $i,j \in \{1,2,3\}$, and any choice of $k+1$ elements from $\{\{v_i,v_j\}, \{\gamma v_i, \gamma v_j\}, \dots, \{\gamma^{n-1}v_i, \gamma^{n-1} v_j\}\}$, there is no $(k-1)$-critical set containing those $k$ elements.
We prove this by induction on $k$.

We have that $\{x_0,y_0\}$ is not contained in a 0-critical set.
We write $x_i = \gamma^i x_0$, and similarly $y_i$.
By the symmetry of $G$, $\{x_i, y_i\}$ is not in a 0-critical set for any $i \in \{0, \dots, n-1\}$. Hence the basis of induction is complete.

Write $X_k(\{x_0, y_0\}) = \big\{\{x_{i_1}, y_{i_1}\}, \dots, \{x_{i_k},y_{i_k}\} : i_1 < \dots < i_k \in \{0, \dots, n-1\}\big\}$.
Assume no $X_j$ is contained in a $(j-1)$-critical set for all $j\leq k$.
Then, Lemma \ref{lem: X_k induction} completes the induction.
Hence, we have shown that there is a pair $\{x_0, y_0\}$ not in a 0-critical set, and for any such pair, the union of any $k$ symmetric copies of that pair is not contained in a $(k-1)$-critical set.
It remains to show is that the 1-reduction does not violate the inequality $|E'| \leq 2|V'|-3$.

We require an analogous reduction to Lemma \ref{lem: X_k induction} with edge-criticality. To this end,  for a given $k$-edge-critical set $W$, we consider the union $\bigcup_{i=0}^{n-1}\gamma^i W$, which is a $C_n$-symmetric set.
In this case it is possible for the order of the group to divide the number of edges in the set, so instead we will show that $\bigcup_{i=0}^{n-1}\gamma^i W$ has $2|\bigcup_{i=0}^{n-1}\gamma^i W|-a$ edges for some $a<2n$. 
Since $\bigcup_{i=0}^{n-1}\gamma^i W$ is symmetric we know $n$ divides $|\bigcup_{i=0}^{n-1}\gamma^i W|$. Hence it must be that $a=n$.

Assume no $X_j$ is contained in a $(j+2)$-edge-critical set for all $j\leq k$.
Then, for a contradiction, suppose $X_{k+1}$ is contained in a $(k+3)$-edge-critical set.
That is, there exists a $W$ containing $k+1$ copies of $\{x_0, y_0\}$ such that $\Bk_W = k+3 = |X| + 3 -1$.
We follow a similar approach to Lemma \ref{lem: X_k induction}, however it is now vital that the appropriate intersections be non-empty.

First take $|X \cap \gamma^{i_2 -i_1} X| < |X| < |X \cup \gamma^{i_2 -i_1} X|$.
By the induction hypothesis, any set containing $\phi^{-1}(X \cap \gamma^{i_2 -i_1} X)$, such as $W\cap \gamma^{i_2 -i_1} W$, is not contained in a $(|X\cap \gamma^{i_2 -i_1} X|+2)$-critical set.
This implies $k_{W\cap \gamma^{i_2 -i_1} W} \geq |X \cap \gamma^{i_2 -i_1} X|+3$.
Lemma \ref{critcapcup} implies that
\begin{equation}\label{eq: W and X}
    \begin{split}
        \Bk_{W \cup \gamma^{i_2 -i_1} W} & = \Bk_W + \Bk_{\gamma^{i_2 -i_1} W} - \Bk_{W \cap \gamma^{i_2 -i_1} W} - d(W, \gamma^{i_2 -i_1} W) \\
        & \leq |X|+3 -1 + |\gamma^{i_2 -i_1} X|+3 -1 - |X \cap \gamma^{i_2 -i_1} X|-3 \\
        & \leq |X \cup \gamma^{i_2 -i_1} X| +3 -2.
    \end{split}
\end{equation}
Hence the critical value for $W \cup \gamma^{i_2 -i_1} W$ is at most $|X \cup \gamma^{i_2 -i_1} X| +3 -2$.
We repeat this process with $\gamma^{i_3 -i_1} X \not\subseteq X \cup \gamma^{i_2 -i_1} X$, and hence take $(X \cup \gamma^{i_2 -i_1} X) \cup \gamma^{i_3 -i_1} X$ and $(X \cup \gamma^{i_2 -i_1} X) \cap \gamma^{i_3 -i_1} X$.
Continuing, we obtain
$$\Bk_{W\cup \gamma^{i_2 -i_1} W \cup \dots \cup \gamma^{i_k -i_1} W} \leq |X\cup \dots \cup \gamma^{i_k -i_1} X| + 3 - k.$$
If $X$ generates $\Z_n$, with the same reasoning as in Lemma \ref{lem: X_k induction}, we can choose an ordering $j_1, \dots, j_n$ of $\Z_n$ such that $$\Bk_{\gamma^{j_1}W\cup \dots \cup \gamma^{j_n} W} \leq |\gamma^{j_1}X\cup \dots \cup \gamma^{j_n} X| + 3 - n+k=k+3.$$
Since $\gamma^{j_1}W\cup \dots \cup \gamma^{j_n} W$ is $c_n$-symmetric, we have a contradiction unless $\Bk_{\gamma^{j_1}W\cup \dots \cup \gamma^{j_n} W} = n$.
Suppose this equality holds, and write $U = \bigcup_{i=0}^{n-1} \gamma^i W$.

\begin{figure}
    \begin{center}
    \begin{tikzpicture}
  [scale=.5]
 
  \coordinate (n101) at (2.5,0);
  \coordinate (n102) at (2.165,1.25);
  \coordinate (n103) at (1.25,2.165);
  \coordinate (n104) at (-1.25,2.165);
  \coordinate (n105) at (-2.165,1.25);
  \coordinate (n106) at (-2.5,0);
  \coordinate (n107) at (-1.25,-2.165);
  \coordinate (n108) at (0,-2.5);
  \coordinate (n109) at (1.25,-2.165);
  \coordinate (n110) at (0.8*4.899,0.8*2.828);
  \coordinate (n111) at (0.8*-4.899,0.8*2.828);
  \coordinate (n112) at (0,-0.8*5.657);

  \coordinate (n201) at (12.5,0);
  \coordinate (n202) at (12.165,1.25);
  \coordinate (n203) at (11.25,2.165);
  \coordinate (n204) at (8.75,2.165);
  \coordinate (n205) at (10-2.165,1.25);
  \coordinate (n206) at (7.5,0);
  \coordinate (n207) at (8.75,-2.165);
  \coordinate (n208) at (10,-2.5);
  \coordinate (n209) at (11.25,-2.165);
  \coordinate (n210) at (0.8*4.899+10,0.8*2.828);
  \coordinate (n211) at (0.8*-4.899+10,0.8*2.828);
  \coordinate (n212) at (10,-0.8*5.657);

  \coordinate (n1) at (1.1*2.5,0);
  \coordinate (n2) at (1.1*2.165,1.1*1.25);
  \coordinate (n4) at (-1.1*1.25,1.1*2.165);
  \coordinate (n5) at (-1.1*2.165,1.1*1.25);
  \coordinate (n7) at (-1.1*1.25,-1.1*2.165);
  \coordinate (n8) at (0,-1.1*2.5);
  \coordinate (n12) at (1.1*2.165+10,1.1*1.25);
  \coordinate (n13) at (1.1*1.25+10,1.1*2.165);
  \coordinate (n15) at (-1.1*2.165+10,1.1*1.25);
  \coordinate (n16) at (-1.1*2.5+10,0);
  \coordinate (n18) at (10,-1.1*2.5);
  \coordinate (n19) at (1.1*1.25+10,-1.1*2.165);

 \draw(0,0) node{$U$}; 
 \draw(10,0) node{$T$};  

 \draw[fill=black] (n101) circle (0.15cm)
        node[right] {$x_0$};
 \draw[fill=black] (n102) circle (0.15cm)
        node[left] {$y_0$};
 \draw[fill=black] (n103) circle (0.15cm)
        node[ above left] {$z_0$};
 \draw[fill=black] (n104) circle (0.15cm);
 \draw[fill=black] (n105) circle (0.15cm);
 \draw[fill=black] (n106) circle (0.15cm);
 \draw[fill=black] (n107) circle (0.15cm);
 \draw[fill=black] (n108) circle (0.15cm);
 \draw[fill=black] (n109) circle (0.15cm);
 \draw[fill=black] (n110) circle (0.15cm)
	    node[right] {$v$};
 \draw[fill=black] (n111) circle (0.15cm)
	    node[left] {$\gamma v$};
 \draw[fill=black] (n112) circle (0.15cm)
	    node[left] {$\gamma^{2} v$};

   \foreach \from/\to in {n101/n110,n102/n110,n103/n110,n104/n111,n105/n111,n106/n111,n107/n112,n108/n112,n109/n112} 
    \draw (\from) -- (\to);

 \draw[fill=black] (n201) circle (0.15cm);
 \draw[fill=black] (n202) circle (0.15cm);
 \draw[fill=black] (n203) circle (0.15cm);
 \draw[fill=black] (n204) circle (0.15cm);
 \draw[fill=black] (n205) circle (0.15cm);
 \draw[fill=black] (n206) circle (0.15cm);
 \draw[fill=black] (n207) circle (0.15cm);
 \draw[fill=black] (n208) circle (0.15cm);
 \draw[fill=black] (n209) circle (0.15cm);
 \draw[fill=black] (n210) circle (0.15cm)
	    node[right] {$v$};
 \draw[fill=black] (n211) circle (0.15cm)
	    node[left] {$\gamma v$};
 \draw[fill=black] (n212) circle (0.15cm)
	    node[left] {$\gamma^{2} v$};

   \foreach \from/\to in {n201/n210,n202/n210,n203/n210,n204/n211,n205/n211,n206/n211,n207/n212,n208/n212,n209/n212} 
    \draw (\from) -- (\to);

   \foreach \from/\to in {n1/n2,n2/n4,n4/n5,n5/n7,n7/n8,n8/n1}
    \draw[dashed] (\from) -- (\to);
   \foreach \from/\to in {n12/n13,n13/n15,n15/n16,n16/n18,n18/n19,n19/n12}
    \draw[dashed] (\from) -- (\to);

\end{tikzpicture}
    \caption{Two $C_n$-symmetric sets $U$ and $T$, with $U$ $n$-critical and $T$ $0$-critical. The case when $n=3$ is shown.}
    \label{fig: U n-crit T 0-crit}
    \end{center}
\end{figure}
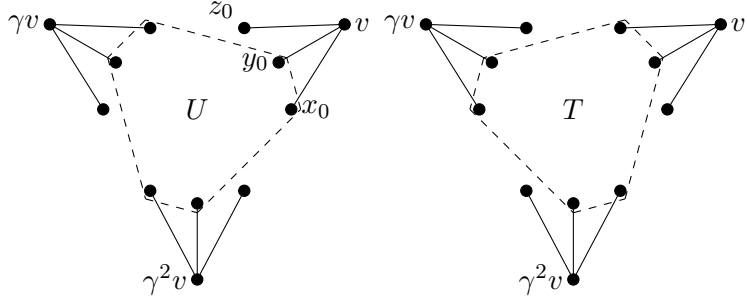

Without loss of generality we may suppose $\{x_0,y_0\}=\{v_1,v_2\}$.
If both edges $v_1v_3, v_2 v_3$ are present, then $\bigcup_{i=0}^{n-1} (\gamma^iW \cup \{z_i, v_i\})$ breaks sparsity. So suppose otherwise.
We will show that one pair from $\{\{v_1, v_3\}, \{v_2, v_3\}\}$ is in neither a 0-critical or a 3-edge-critical set.
Take a pair from $\{\{v_1, v_3\}, \{v_2, v_3\}\}$,  say $\{y_0,z_0\}$.
If $\{y_0,z_0\}$ is in a 0-critical set $T_0$, then $T = \bigcup_{i=0}^{n-1}\gamma^i T_0$ is 0-critical.
We have already shown that $U$ is not $p$-critical for $p\in \{0, \dots, n-1\}$, hence $n = \Bk_U \geq k_U \geq n$, therefore $U$ must be $n$-critical and $i_L(U)=0$.
Then, 
$$k_{U\cup T} \leq k_U + k_T - k_{U\cap T} = n -k_{U\cap T}.$$
Since $i_L(U)=0$, $i_L(U\cap T)=0$, so $k_{U \cap T} = \Bk_{U \cap T} \geq 3$. This gives $k_{U \cup T} \leq n-3$.
Then, $\bigcup_{i=0}^{n-1}\gamma^i v \cup U \cup T$ violates the sparsity of $G$ (see Figure \ref{fig: U n-crit T 0-crit} for an illustration).
Similarly, if $x_0z_0 \notin E$, then $\{x_0, z_0\}$ is not in a 0-critical set, and we can apply the inductive argument from the beginning of the proof to deduce that no $k$ copies of either pair is in a $(k-1)$-critical set.

Now suppose $\{y_0, z_0\}$ is in a 3-edge-critical set $T_0$.
If $x_0z_0 \in E$, then $$\Bk_{U\cup T_0} = \Bk_U + \Bk_{T_0} - \Bk_{U\cap T_0} - d(U, T_0) \leq n + 3 - 2 -1 = n.$$ 
Repeating this with $T_1 = \gamma T_0$, so $\Bk_{U\cup T_0 \cup T_1} \leq n$, until $T_{n-1} = \gamma^{n-1} T_0$ gives $\Bk_{U\cup T_0 \cup \dots \cup T_{n-1}} \leq n$.
This union with $\bigcup_{i=0}^{n-1}\{\gamma^i{v}\}$ breaks sparsity of $G$.
We note that the above contradiction would hold with $x_0z_0 \notin E$ and $\Bk_{U\cap T_0} \geq 3$.
Therefore assume $x_0z_0 \notin E$ and $\Bk_{U\cap T_0} = 2$.
Similar to the above, we arrive at a contradiction if $\{x_0, z_0\}$ is in a 3-edge-critical $S_0$ unless $\Bk_{U\cap S_0} = 2$.
Then 
$$\Bk_{S_0 \cup T_0} = \Bk_{S_0} + \Bk_{T_0} - \Bk_{S_0 \cap T_0} - d(S_0, T_0) \leq 3 + 3 - 2 = 4.$$
By definition, edge-criticality equal to $2$ implies the vertex set is a singleton. Hence $U \cap T_0 = \{y_0\}$ and $U \cap S_0 = \{x_0\}$.
Then $U \cap (S_0 \cup T_0) = \{x_0, y_0\}$, and since $x_0 y_0 \notin E$, $\{x_0, y_0\}$ is 4-edge-critical (as depicted in Figure \ref{fig: U S_0 T_0}).
\begin{figure}
    \centering
    \begin{tikzpicture}
  [scale=.5,auto=left]
  
  \coordinate (n1) at (0,3);
  \coordinate (n2) at (3,3);
  \coordinate (n3) at (4.5,0.5);
  \coordinate (n4) at (4.5,5.5);
  \coordinate (n5) at (9.5,0.5);
  \coordinate (n6) at (9.5,5.5);
  \coordinate (n7) at (11,3);
  \coordinate (n8) at (14,3);
  \coordinate (n9) at (7,3);
  \coordinate (n10) at (2,1.15);
  \coordinate (n11) at (2,4.85);

 \draw[dashed] (4,-1) rectangle (10,7);
 \draw[fill=black] (n1) circle (0.15cm)
	    node[above left] {$z_0$};
 \draw[fill=black] (n2) circle (0.15cm)
	    node[above left] {$v$};
 \draw[fill=black] (n3) circle (0.15cm)
	    node[below left] {$x_0$};
 \draw[fill=black] (n4) circle (0.15cm)
	    node[above left] {$y_0$};
 \draw (n9) node {$U$};
 \draw (n10) node {$S_0$};
 \draw (n11) node {$T_0$};
 \draw[dashed] (5,5.5) arc (30.964:210.964:2.9155);
 \draw[dashed] (5,5.5) -- (0,2.5);
 \draw[dashed] (0,3.5) arc (149.036:329.036:2.9155);
 \draw[dashed] (5,0.5) -- (0,3.5);

  \foreach \from/\to in {n1/n2,n2/n3,n2/n4} 
    \draw (\from) -- (\to);

\end{tikzpicture}
    \caption{$C_n$-symmetric and $n$-edge-critical $U$ with $3$-edge-critical sets $S_0$ and $T_0$.}
    \label{fig: U S_0 T_0}
\end{figure}
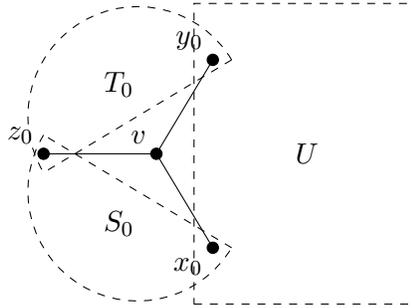

We have
$$\Bk_{U\cup (S_0 \cup T_0)} \leq \Bk_U + \Bk_{S_0 \cup T_0} - \Bk_{U\cap (S_0 \cup T_0)} \leq n + 4 - 4 = n.$$ 
Repeating with $(S_i, T_i) = (\gamma^i S_0, \gamma^i T_0)$ for $i =1, \dots n-1$ implies that $\bigcup_{i=0}^{n-1}(S_i \cup T_i) \cup U$ is $n$-edge-critical. Then adding $v, \dots, \gamma^{n-1}v$ breaks sparsity of $G$.
Therefore one of the pairs $\{x_0, z_0\}$ and $\{y_0, z_0\}$ are not contained in a $3$-edge-critical set.
Without loss of generality, say that it is $\{x_0, z_0\}$.
We can now build an inductive argument, assuming $q$ copies of $\{x_0, z_0\}$ are not contained in a $(q+2)$-critical set for $1\leq q \leq k_1$.
As before with $\{x_0, y_0\}$, suppose $k_1 + 1$ copies of $\{x_0, z_0\}$ are contained in a $(k_1 + 2)$-edge-critical set $R$.
As with $W$, $\gamma^{h_1}R\cup \dots \cup \gamma^{h_n} R$ is $C_n$-symmetric. Hence we have a contradiction unless $\Bk_{\gamma^{h_1}R\cup \dots \cup \gamma^{h_n} R} = n$ for some ordering $h_1, \dots, h_n$ of $\Z_n$.
Recall that $U = \gamma^{j_1}W\cup \dots \cup \gamma^{j_n} W$ and put $R^* =\gamma^{h_1}R\cup \dots \cup \gamma^{h_n} R$.
Then $R^*$ is $C_n$-symmetric.
For any set the edge-criticality is at least 2, hence by sparsity and symmetry $\Bk_{U \cap R^*} \geq n$.
Therefore, 
$$\Bk_{U \cup R^*} \leq \Bk_{U} + \Bk_{R^*} - \Bk_{U \cap R^*} \leq n+n-n = n,$$ 
which, on adding $\bigcup_{i=0}^{n-1}\{\gamma^i v\}$, violates the sparsity of $G$.
This completes the induction for one of $\{x_0, y_0\}$ and $\{x_0, z_0\}$.

If $X$ generates a subgroup of $\mathbb{Z}_n$, let $m$ be such that $m \cdot |\langle X \rangle| = n$.
Then, with $l=|\langle X \rangle|\geq k+1$, there is an ordering $j_1, \dots, j_l$ of $\langle X \rangle$ such that $$\Bk_{\gamma^{j_1}W\cup \dots \cup \gamma^{j_l} W} \leq |\gamma^{j_1}X\cup \dots \cup \gamma^{j_l} X| + 3 - l+k=k+3.$$
Let $W^* = \gamma^{j_1}W\cup \dots \cup \gamma^{j_l} W$. By construction we have $$W^* \cup \gamma W^* \cup \dots \cup \gamma^{m-1}W^* = W \cup  \gamma W \cup \dots \cup \gamma^{n-1}W.$$
Whenever $n\geq 3$ is odd, we have that $$m (k+3) = m(k+1) + 2m \leq ml +2m =n+2m <2n.$$
Hence, 
$$\Bk_{W \cup \dots \cup \gamma^{n-1}W} = m\Bk_{\gamma^{j_1}W\cup \dots \cup \gamma^{j_l} W} < 2n,$$ 
and since $W \cup \dots \cup \gamma^{n-1}W$ is $C_n$-symmetric we may repeat the argument in the paragraph above to obtain a contradiction.
\end{proof}

It was only in the final paragraph of the proof where the proof does not apply to an arbitrary cyclic group.
Consider a $C_{2m}$-tight graph $G$, and $X_2 = \{\{x_0,y_0\},\{x_m,y_m\}\}$ contained in a 4-edge-critical set $W$.
Then assuming each of the sets are disjoint, $W \cup \dots \cup \gamma^{m-1}W$ is $4m=2n$-edge-critical.
A different approach is therefore required in these groups.

We can now present our main results.

\begin{thm}\label{thm: codd recursion}
Let $n$ be a positive odd integer.
A graph $(G,\phi)$ is $C_n$-tight if and only if every $C_n$-symmetrically connected component of $G$ can be generated from $(\mathcal{P}_n,\phi_n)$ or $(LC_n,\psi_n)$ by symmetrised 0-extensions and 1-extensions.
\end{thm}

\begin{proof}
Any $C_n$-symmetrically connected component generated from one of the base graphs by symmetrised 0-extensions and 1-extensions is $C_n$-tight. For the converse, we show by induction that any $C_n$-tight graph $G$ can be generated from our base graphs.
We may assume by Lemma \ref{lem: components} that $G$ is $C_n$-symmetrically connected.
Suppose the induction hypothesis holds for all graphs with $|V| < m$.
Now let $|V|= m$ and suppose $G$ is not isomorphic to one of the base graphs in Figure \ref{fig: c_5 base}.
For a tight graph, $2\leq \delta(G)\leq 4$.

A degree 2 vertex is reducible by Lemma \ref{lem: c_n two edges redu}.
A degree 3 vertex can have a loop and an edge incident to it or be adjacent to three vertices. The former is reducible by Lemma \ref{lem: c_n two edges redu} and the latter by Lemma \ref{lem: c_n 3edges redu}.
When $\delta(G)=4$, it can be shown (see the proof of Theorem \ref{thm:c2recursion}) that there exists a vertex $w\in V$ incident to a loop and adjacent to two vertices.
Then $w$ is reducible by Lemma \ref{lem: c_n 2 edges 1 loop redu}, completing the proof.
\end{proof}

\begin{thm}\label{thm:maincodd}
Let $n$ be a positive odd integer. A graph $(G,\phi)$ is $C_n$-isostatic if and only if it is $C_n$-tight.
\end{thm}

\begin{proof}
    Since $C_n$-isostatic graphs are tight, necessity follows from Theorem \ref{thm: fixed 2dim lin}.
    In Lemma \ref{lem: base rigid}, the base graphs $(\mathcal{P}_n,\phi_n)$ and $(LC_n,\psi_n)$ ($n=5$ depicted in Figure \ref{fig: c_5 base}) are $C_n$-isostatic.
    Hence the sufficiency follows from Theorem \ref{thm: codd recursion} and Lemmas \ref{0-ext rigid} and \ref{1-ext rigid} by induction on $|V|$.
\end{proof}

\section{Concluding remarks} \label{sec:final}

An immediate consequence of  Theorems \ref{thm:mainc2} and \ref{thm:maincodd} is that there are efficient, deterministic algorithms for determining whether a given graph is $C_n$-isostatic (for $n=2$ or $n\geq 3$ odd) since the sparsity counts can be checked using the standard pebble game algorithm \cite{Hen,ls} and  the additional symmetry conditions for the number of fixed vertices and edges can be checked in constant time, from the group action $\phi$.

The next obvious challenge would be to extend the characterisations in Theorems \ref{thm:mainc2} and \ref{thm:maincodd} to deal with the remaining symmetry groups of the plane, as described in Theorem \ref{thm: fixed 2dim lin}.  
For the case of even order cyclic symmetry, the key difficulty to overcome is the one described after the proof of Lemma~\ref{lem: c_n 3edges redu}.
For the case of reflection symmetry there are many new subcases to the analysis of the extension/reduction operations used in this paper resulting from fixed edge, loop and vertex counts being arbitrarily large.
The count given in Theorem \ref{thm: fixed 2dim lin} along with $v_\sigma\geq \max\{l_{\sigma,+},l_{\sigma,-}\}$ must be preserved by any such move.
Further discussion can be found in \cite{Wall}.

If $C_n$ acts freely on the vertices, edges and loops, we \emph{conjecture} that an infinitesimally rigid $C_n$-symmetric linearly constrained framework in $\mathbb{R}^2$ will always have a spanning isostatic subframework with the same symmetry.  
If so, generic infinitesimal rigidity of  $C_n$-symmetric linearly constrained frameworks in $\mathbb{R}^2$, where $C_n$ acts freely, can be characterised in terms of symmetric isostatic subframeworks.  This is in general not true; Figure \ref{fig: c3rigidcounter} provides a small counterexample.

\begin{figure}[ht]
    \begin{center}
    \begin{tikzpicture}
  [scale=.4]

  \coordinate (n1) at (0,3);
  \coordinate (n2) at (-0.866*3,-1.5);
  \coordinate (n3) at (0.866*3,-1.5);
  \coordinate (n0) at (0,0);
  
 \draw[fill=black] (n1) circle (0.15cm);
 \draw[fill=black] (n2) circle (0.15cm);
 \draw[fill=black] (n3) circle (0.15cm);
 \draw[fill=black] (n0) circle (0.15cm);
 
\draw[scale = 4] (0,0)  to[in=-10,out=70,loop] (0,0);
\draw[scale = 4] (0,0)  to[in=110,out=190,loop] (0,0);
\draw[scale = 4] (0,0)  to[in=-130,out=-50,loop] (0,0);
\draw[scale = 4] (0,0.75)  to[in=50,out=130,loop] (0,0);
\draw[scale = 4] (-0.2165*3,-0.375)  to[in=170,out=250,loop] (0,0);
\draw[scale = 4] (0.2165*3,-0.375)  to[in=-70,out=10,loop] (0,0);

   \foreach \from/\to in {n0/n1,n0/n2,n0/n3} 
    \draw (\from) -- (\to);
 
\end{tikzpicture}
    \caption{A $C_3$-symmetric rigid graph with no spanning isostatic $C_3$-symmetric subgraph.}
    \label{fig: c3rigidcounter}
    \end{center}
\end{figure}
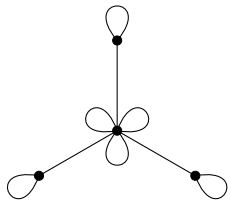

It is conceivable that one may be able to adapt the analysis of \cite{JNT}, in combination with the techniques developed in this article, to extend our characterisations to symmetric linearly constrained frameworks in arbitrary dimension, under appropriate hypotheses on the affine subspace constraints.
In particular, in \cite{Wall}, the necessary conditions for some isostatic $\tau(\Gamma)$-symmetric linearly constrained frameworks in $\R^d$ have been shown.

It would also be interesting to extend our results to symmetric but non-generic linear constraints. In the plane this would give rise to potential analogues of a result of Katoh and Tanigawa \cite{KatTan}. In 3-dimensions this could extend results for surfaces, as obtained in \cite{NOP,NOP14}, for example to allow a small number of `free' vertices.

\section*{Acknowledgements}

AN was partially supported by EPSRC grants EP/W019698/1 and EP/X036723/1.

\bibliographystyle{abbrv}

\begin{thebibliography}{40}

\bibitem{Abb} T. Abbott,
\emph{Generalizations of Kempe's universality theorem},
Master's thesis, Massachusetts Institute of Technology (2008).


\bibitem{AR} L. Asimow and B. Roth,
\emph{The rigidity of graphs},
Transactions of the American Mathematical Society,
245 (1978) 279--289.



\bibitem{cfgsw} R. Connelly, P. Fowler, S. Guest, B. Schulze and W. Whiteley, \emph{When is a symmetric pin-jointed framework isostatic?}, International Journal of Solids and Structures 46 (2009), 762--773.

\bibitem{cg} R. Connelly and S.D. Guest, \emph{Frameworks, Tensegrities, and Symmetry}, Cambridge University Press, 2022.

\bibitem{CGJN} J. Cruickshank, H. Guler, B. Jackson and A. Nixon,
\emph{Rigidity of linearly constrained frameworks}, 
International Mathematics Research Notices, 2020:12 (2020) 3824--3840.

\bibitem{EJNSTW} Y. Eftekhari, B. Jackson, A. Nixon, B. Schulze, S. Tanigawa and W. Whiteley, \emph{Point-hyperplane frameworks, slider joints, and rigidity preserving transformations}, Journal of Combinatorial Theory: Series B, 135 (2019) 44--74.

\bibitem{FGsymmax} P.W. Fowler and S.D. Guest, \emph{A symmetry extension of Maxwell's rule for rigidity of frames}, International Journal of Solids and Structures, 37 (2000) 1793--1804.


\bibitem{GJN}
H. Guler, B. Jackson and A. Nixon,
\emph{Global rigidity of 2D linearly constrained frameworks}, 
International Mathematics Research Notices, 2021:22 (2021) 16811--16858.

\bibitem{Hen} B. Hendrickson and D. Jacobs, \emph{An Algorithm for two-dimensional Rigidity Percolation: the Pebble Game}, Journal of Computational Physics, 137 (1997) 346--365.


\bibitem{JNT} B. Jackson, A. Nixon and S. Tanigawa, \emph{An improved bound for the rigidity of linearly constrained frameworks}, SIAM Journal on Discrete Mathematics, 35:2 (2021) 928--933.


\bibitem{KatTan} N. Katoh and S. Tanigawa, \emph{Rooted-tree decompositions with matroid constraints and the infinitesimal rigidity of frameworks with boundaries}, SIAM Journal on Discrete Mathematics,
27:1 (2013) 155--185.

\bibitem{laman} G. Laman, \emph{On graphs and rigidity of plane skeletal structures}, Journal of Engineering Mathematics, 4 (1970) 331--340.

\bibitem{ls} A. Lee and I. Streinu, \emph{Pebble Game Algorithms and Sparse Graphs}, Discrete Mathematics, 308, 8, (2008) 1425--1437.

\bibitem{NOP} A. Nixon, J.C. Owen, and S.C. Power, 
\emph{Rigidity of frameworks supported on surfaces}, SIAM Journal on Discrete Mathematics, 26:4 (2012) 1733--1757.

\bibitem{NOP14} A.~Nixon,  J.~Owen and S.C.~Power, \emph{A characterization of generically rigid frameworks on surfaces of revolution}, SIAM Journal on Discrete Mathematics, 28:4 (2014) 2008--2028.


\bibitem{NSW} A. Nixon, B. Schulze and J. Wall, \emph{Rigidity of symmetric frameworks on the cylinder}, preprint, arXiv:2210.06060, 2022.


\bibitem{PG} H. Pollaczek-Geiringer, \emph{Uber die Gliederung ebener Fachwerke}, Zeitschrift fur Angewandte Mathematik und Mechanik (ZAMM), 7 (1927) 58-72.

\bibitem{schulze} B.~Schulze, \emph{Symmetric versions of Laman's theorem}, Discrete and Computational Geometry, 44:4 (2010) 946--972.

\bibitem{BS4} B.~Schulze, \emph{Symmetric Laman theorems for the groups $C_2$ and $C_s$},
The Electronic Journal of Combinatorics, 17:1 (2010) R154, 1--61.

\bibitem{BS2}
B.~Schulze, \emph{Block-diagonalized rigidity matrices of symmetric frameworks and
  applications}, Contributions to Algebra and Geometry, 51:2 (2010) 427--466.
  
\bibitem{ST} B. Schulze and S. Tanigawa, \emph{Infinitesimal rigidity of symmetric bar-joint frameworks}, SIAM Journal on Discrete Mathematics, 29:3 (2015) 1259--1286. 
  

\bibitem{SW2} B. Schulze and W. Whiteley,  \emph{Rigidity of symmetric frameworks}, Handbook of Discrete and Computational Geometry,  C.D. Toth, J. O'Rourke, J.E. Goodman, editors, Third Edition, Chapman \& Hall CRC, 2018.

\bibitem{Serre}
J.-P. Serre,  \emph{Linear Representations of Finite Groups}, Springer-Verlag, 1977 

\bibitem{STh} I. Streinu and L. Theran,
\emph{Slider-pinning rigidity: a Maxwell-Laman-type theorem},
Discrete and Computational Geometry
44:4 (2010) 812--837.

\bibitem{Tetal} 
L. Theran, A. Nixon, E. Ross, M. Sadjadi, B. Servatius and M. Thorpe, \emph{Anchored boundary conditions for locally isostatic networks},
Physical Review E, 92:5 (2015) 053306.

\bibitem{Wall}
J. Wall, \emph{Rigidity of symmetric frameworks on surfaces}, PhD thesis, Lancaster University, 2024.

\end{thebibliography}
\def\lfhook#1{\setbox0=\hbox{#1}{\ooalign{\hidewidth
  \lower1.5ex\hbox{'}\hidewidth\crcr\unhbox0}}}

\end{document}